\documentclass[a4paper,10pt]{amsart}

\usepackage[english]{babel}
\usepackage[utf8]{inputenc}
\usepackage{ulem} 

\usepackage{amssymb}
\usepackage{amsmath}
\usepackage{amsthm}
\usepackage{bbm}

\usepackage{enumerate}
\usepackage{tikz}
\usepackage{marginnote}
\usepackage{mathtools}

\newcommand{\bbC}{\mathbb{C}}

\newcommand{\bbN}{\mathbb{N}}

\newcommand{\bbQ}{\mathbb{Q}}
\newcommand{\bbR}{\mathbb{R}}

\newcommand{\bbT}{\mathbb{T}}

\newcommand{\bbZ}{\mathbb{Z}}

\newcommand{\calA}{\mathcal{A}}
\newcommand{\calB}{\mathcal{B}}

\newcommand{\calG}{\mathcal{G}}

\newcommand{\calI}{\mathcal{I}}

\newcommand{\calK}{\mathcal{K}}
\newcommand{\calL}{\mathcal{L}}
\newcommand{\calM}{\mathcal{M}}

\newcommand{\calT}{\mathcal{T}}
\newcommand{\calU}{\mathcal{U}}

\DeclareMathOperator{\id}{id}
\DeclareMathOperator{\one}{\mathbbm{1}}

\DeclareMathOperator{\dist}{dist}
\newcommand{\argument}{\mathord{\,\cdot\,}}

\DeclareMathOperator{\lin}{lin}

\DeclareMathOperator{\Ball}{B}

\newcommand{\ue}{\mathrm{e}}
\newcommand{\ui}{\mathrm{i}}

\DeclarePairedDelimiter{\norm}{\lVert}{\rVert}
\DeclarePairedDelimiter{\modulus}{\lvert}{\rvert}

\newcommand{\spec}{\sigma}
\newcommand{\Res}{\mathcal{R}}

\newcommand{\pnt}{{\operatorname{pnt}}}

\newcommand{\sgInftyON}[1]{{#1}_\infty}

\theoremstyle{definition}
\newtheorem{definition}{Definition}[section]
\newtheorem{remark}[definition]{Remark}

\newtheorem{example}[definition]{Example}
\newtheorem{examples}[definition]{Examples}
\newtheorem*{research_questions}{Research Questions}

\theoremstyle{plain}
\newtheorem{proposition}[definition]{Proposition}
\newtheorem{lemma}[definition]{Lemma}
\newtheorem{theorem}[definition]{Theorem}
\newtheorem{corollary}[definition]{Corollary}

\numberwithin{equation}{section}

\begin{document}

\normalem

\title[Uniform convergence]{Uniform convergence of operator semigroups without time regularity}

\author{Alexander Dobrick}
\address{Alexander Dobrick, Arbeitsbereich Analysis, Christian-Albrechts-Universit\"at zu Kiel, Ludewig-Meyn-Str.\ 4, 24098 Kiel, Germany}
\email{dobrick@math.uni-kiel.de}
\author{Jochen Gl\"uck}
\address{Jochen Gl\"uck, Fakultät für Informatik und Mathematik, Universität Passau, Innstraße 33, D-94032 Passau, Germany}
\email{jochen.glueck@uni-passau.de}
\subjclass[2010]{47D03; 47D06; 35K40; 35B40}
\keywords{Coupled heat equations; Schroedinger semigroups; matrix potential; long-term behaviour; semigroup representation; Jacobs--de Leeuw--Glicksberg theory; semigroup at infinity}
\date{\today}
\begin{abstract}
	When we are interested in the long-term behaviour of solutions to linear evolution equations, a large variety of techniques from the theory of $C_0$-semigroups is at our disposal. However, if we consider for instance parabolic equations with unbounded coefficients on $\bbR^d$, the solution semigroup will not be strongly continuous, in general. For such semigroups many tools that can be used to investigate the asymptotic behaviour of $C_0$-semigroups are not available anymore and, hence, much less is known about their long-time behaviour. 
	
	Motivated by this observation, we prove new characterisations of the operator norm convergence of general semigroup representations -- without any time regularity assumptions -- by adapting the concept of the ``semigroup at infinity'', recently introduced by M.~Haase and the second named author. Besides its independence of time regularity, our approach also allows us to treat the discrete-time case (i.e., powers of a single operator) and even more abstract semigroup representations within the same unified setting. 
	
	As an application of our results, we prove a convergence theorem for solutions to systems of parabolic equations with the aforementioned properties.
\end{abstract}

\maketitle

\section{Introduction}

The purpose of this article is to study uniform convergence to equilibrium for linear operator semigroups as time tends to infinity. For powers $T^n$ of a single operator $T$, it is well-known that this kind of long-time behaviour can characterised by spectral properties of $T$. For $C_0$-semigroups, the situation is more subtle, but has been extensively studied throughout the literature. We refer, for instance, to the classical references \cite{Neerven1996}, \cite[Chapter~V]{Engel2000}, \cite{Emelyanov2007}, \cite{Eisner2010} and \cite[Chapter 14]{Batkai2017} for more information. Still, the current state of the art leaves following issues open:

\begin{research_questions}
	\begin{enumerate}[(1)]
		\item For semigroups in the continuous time interval $[0,\infty)$, much of the known theory deals with the case of $C_0$-semigroups. This special case is very useful for many applications, for instance in the analysis of partial differential equations. On the other hand, there are important PDEs whose solution semigroup is not strongly continuous; this situation occurs, for instance, for parabolic equations with unbounded coefficients on $\bbR^d$ (see Section~\ref{section:application-coupled-parabolic-equations}). Therefore, the treatment of such examples requires a theory which is capable of efficiently handling semigroups that are not strongly continuous.
		
		\item The quest for a most cohesive and clear theory suggests that we should also seek for methods which help us to treat the discrete-time case (i.e., powers of a single operator) and the continuous-time case (i.e., semigroups indexed over the time interval $[0,\infty)$) within a unified theoretical framework. This was also a major guideline in \cite{GerlachConvPOS} and \cite{Glueck2019}.  
	\end{enumerate}
\end{research_questions}

\subsection*{Contributions}

Our answer to the issues mentioned above is as follows: inspired by the classical Jacobs--de Leeuw--Glicksberg (JdLG) theory and its success in the study of the long-time behaviour of strongly or weakly compact operator semigroups, we show how a similar idea can be used to study convergence of semigroups with respect to the operator norm. This is not a straightforward task due to the following obstacle: if $(T_s)_{s \in [0,\infty)}$ is an operator semigroup on a Banach space, even the local orbits
\begin{align*}
	\{T_s\colon s \in [0,s_0]\}
\end{align*}
will typically not be (relatively) compact with respect to the operator norm. Therefore, in order to obtain compactness -- and to thus employ typical JdLG arguments --, we restrict our attention to the behaviour of the semigroup ``at infinitely large times''. This idea leads us to what we call the \emph{semigroup at infinity}. This concept was recently developed in \cite{Glueck2019} in order to study strong convergence of semigroups; here, we adapt it to the operator norm topology -- and as it turns out, this completely resolves the issues that the local orbits of semigroups are not operator norm compact, in general. 

Here is an outline of our general strategy:
\begin{itemize}
	\item In Section~\ref{section:semigroup-representations-and-the-semigroup-at-infinity} we study general representations $(T_s)_{s \in S}$ of commutative semigroups $S$ on Banach spaces. To each such representation we assign a \emph{semigroup at infinity}. Under appropriate assumptions, this yields a splitting of the semigroup into a ``stable part'' that converges to $0$, and a ``reversible part'' that extends to a compact group. 
	
	\item In Section~\ref{section:triviality-of-compact-operator-groups} we give sufficient criteria for compact operator groups to be trivial.
	
	\item By combining the aforementioned results, we finally obtain various criteria for operator norm convergence of semigroups in Section~\ref{section:operator-norm-convergence-of-semigroups}.
\end{itemize}

The following result demonstrates what can be shown by our methods for semigroups indexed over the time interval $[0,\infty)$ without any continuity assumption. For undefined terminology about operator semigroups we refer to the beginning of Section~\ref{section:semigroup-representations-and-the-semigroup-at-infinity}. AM-spaces are a class of Banach lattices that are, for instance, described in \cite[Section~II.7]{Schaefer1974}; here we only mention that the space $C_b(X;\bbR)$ of all bounded continuous real-valued functions on a topological space $X$ is an example of an AM-space.

\begin{theorem} \label{thm:introduction}
	Let $E$ be an AM-space (over the scalar field $\bbR$), or a real-valued $L^p$-space over an arbitrary measure space for $p \in [1,\infty] \setminus \{2\}$. For each contractive operator semigroup $(T_s)_{s \in [0,\infty)}$ on $E$ the following assertions are equivalent:
	\begin{enumerate}[\upshape (i)]
		\item $T_s$ converges with respect to the operator norm to a finite rank projection as $s \to \infty$.
		\item There exists a time $s_0 \in [0,\infty)$ such that $T_{s_0}$ is quasi-compact.
	\end{enumerate}
\end{theorem}

Here, \emph{contractive} means that $\norm{T_s} \le 1$ for all $s \in [0,\infty)$. We prove this theorem at the end of Subsection~\ref{subsection:convergence-under-divisibility-conditions}. In Section~\ref{section:application-coupled-parabolic-equations} we show how the theorem can be applied to study the long-time behaviour of certain parabolic systems on $\bbR^d$. In the appendix we recall a few facts on poles of operator resolvents and on the behaviour of nets in metric spaces.

\begin{remark}
	We mentioned two research questions at the beginning of the introduction. Theorem~\ref{thm:introduction} demonstrates that our methods yield non-trivial answers to question~(1). But the theorem also allows for an interesting insight concerning question~(2):
	
	The powers of a two-dimensional permutation matrix show that the conclusion of the theorem fails for semigroups indexed over $\bbN_0$ rather than $[0,\infty)$. The fact that we treat general semigroup representations enables us to understand this phenomenon on a very conceptual level: it is caused by the different algebraic properties of the semigroups $([0,\infty),+)$ and $(\bbN_0,+)$ -- the first one is \emph{essentially divisible}, while the second one is not. We refer to Subsection~\ref{subsection:convergence-under-divisibility-conditions} and in particular to Remark~\ref{rem:continuous-vs-discrete-time} for details.
\end{remark}

\subsection*{Related literature}

Despite the prevalence of $C_0$-semigroups, semigroups with weaker continuity assumptions occur on many occasions in the literature. The time regularity properties one encounters vary from strong continuity on $(0,\infty)$ (see e.g.\ \cite{Arendt2016, Arendt2018} for two applications) to such concepts as bi-continuity \cite{Kuehnemund2003} and continuity on norming dual pairs \cite{Kunze2009}.

Strong convergence for semigroups that are not $C_0$ has been recently studied in the papers \cite{GerlachLB}, \cite{GerlachConvPOS} and \cite{Glueck2019}; the latter of them is closely related to the approach that we present here. One of the few classical results about operator norm convergence of semigroups that are not $C_0$ is a theorem of Lotz about quasi-compact positive semigroups on Banach lattices \cite[Theorem~4]{Lotz1986}; we generalise this result in Corollary~\ref{cor:quasi-compact-convergence-banach-lattice} below.

In \cite{Dobrick2020} the semigroup of infinity is used to investigate the long-term behaviour of semigroups associated to transport processes on infinite networks with $L^\infty$-state spaces.

\subsection*{Notation and Terminology.} 

We denote the complex unit circle by $\bbT$. All Banach spaces in this paper can be either real or complex, unless otherwise specified. To clarify whether the elements of certain function spaces are assumed to be real- or complex-valued we use notation such as $L^p(\Omega,\mu; \bbR)$ and $L^p(\Omega,\mu;\bbC)$, etc.

Let $E,F$ be Banach spaces (over the same scalar field). We endow the space $\calL(E;F)$ of bounded linear operators from $E$ to $F$ with the operator norm topology throughout; moreover, we use the abbreviation $\calL(E) \coloneqq \calL(E;E)$.
For a set $\calM \subseteq \calL(E)$ and a closed subspace $U \subseteq E$ that is invariant under all operators in $\calM$, we use the notation
\begin{align*}
	\calM|_U := \{T|_U: \, T \in \calM\} \subseteq \calL(U).
\end{align*}
The dual Banach space of $E$ will be denoted by $E'$. If the underlying scalar field is complex, the spectrum of a linear operator $A\colon E \supseteq D(A) \to E$ will be denoted by $\sigma(A)$; for $\lambda \in \bbC \setminus \sigma(A)$, the resolvent of $A$ at $\lambda$ is denoted by $\Res(\lambda,A) \coloneqq (\lambda - A)^{-1}$. Further, the point spectrum of $A$ will be denoted by $\sigma_\pnt(A)$. If the underlying scalar field of $E$ is real, the spectrum and the point spectrum of an operator $A$ are defined as the spectrum and the point spectrum of the canonical extension of $A$ to any complexification of $E$. 

Basic terminology for semigroup representations is introduced at the beginning of the next section.

\section{Semigroup representations and the semigroup at infinity} \label{section:semigroup-representations-and-the-semigroup-at-infinity}

In this section we develop a general framework to analyse whether an operator semigroup converges with respect to the operator norm as time tends to infinity. The most important situation that occurs in applications is that the semigroup contains a quasi-compact operator, and this situation will also be one of our main interests (though not our only interest). In the case of $C_0$-semigroups, a rather complete description of the long-term behaviour in the case of quasi-compactness can be found in \cite[Section~V.3]{Engel2000} (and for more general aspects of the long-term behaviour of $C_0$-semigroups we refer for instance to \cite[Chapter~V]{Engel2000} and \cite{Eisner2010}). However, as it has become apparent in the preceding sections, the case of $C_0$-semigroups is not always sufficient and, as explained in the introduction, we do not wish to develop an individual convergence theory for each different type of time regularity that might occur in applications. Thus, we stick to the other extreme and develop a single theory that does not assume any time regularity at all.

This goal being set, it is just consequent to leave the restricted setting of semigroups of the type $(T_s)_{s \in [0,\infty)}$, and to consider operator representations of general commutative semigroups $(S,+)$ instead. This allows us to also treat the time-discrete case $(T^n)_{n \in \bbN_0}$ and, for instance, the case of multi-parameter semigroups within our one theory. Moreover, it allows for some interesting theoretical observations in the spirit of \cite{GerlachConvPOS} and \cite{Glueck2019}.

Our approach is based on the celebrated Jacobs--de Leeuw--Glicksberg (JdLG) theory which applies abstract results on (semi-)topological semigroups to the more concrete situation of operator semigroups, and we combine this with the construction of a \emph{semigroup at infinity} which is inspired by \cite{Glueck2019}. In this context, we find it also worthwhile to mention that there exist other quite abstract approaches to general operator semigroups, too, that do not rely on JdLG theory (see for instance \cite{Gao2014}); however, we will stick to JdLG theory in this paper.

\subsection{Setting} \label{subsection:general-semigroup-setting}

Throughout the rest of this paper, let $(S,+)$ be a commutative semigroup with neutral element $0$ (i.e., in a more algebraic language, $(S,+)$ is a commutative monoid). We define a reflexive and transitive relation (i.e., a \emph{pre-order}) $\le$ on $S$ by setting
\begin{align*}
	s \le t \quad \text{if and only if} \quad \text{there exists } r \in S \text{ such that } t = s+r
\end{align*}
for $s,t \in S$. Note that $S$ is directed with respect to the pre-order $\le$ since we have $s,t \le s+t$ for all $s,t \in S$.

A \emph{representation} of $S$ on a Banach space $E$ is any mapping $T\colon S \to \calL(E)$ that satisfies
\begin{align*}
T(0) = \id_E \quad \text{and} \quad T(s + t) = T(s) T(t) \qquad \text{for all } t, s \in S.
\end{align*}
In the following, we will often use the index notation $T_s$ instead of $T(s)$ and call $(T_s)_{s \in S}$ an \emph{operator semigroup} on $E$. 

Let $(T_s)_{s \in S}$ be an operator semigroup on $E$, and assume that the underlying scalar field of $E$ is $\bbC$. A function $\lambda \colon S \to \bbC$ is called an \emph{eigenvalue} of $(T_s)_{s \in S}$ if there exists a non-zero vector $x \in E$ such that
\begin{align*}
T_s x = \lambda_s x \qquad \text{for all } s \in S;
\end{align*}
in this case, the vector $x$ is called a corresponding \emph{eigenvector}. Note that an eigenvalue $\lambda = (\lambda_s)_{s \in S}$ is always a representation of $(S,+)$ on the space $\bbC$. Moreover, we call an eigenvalue $\lambda = (\lambda_s)_{s \in S}$ \emph{unimodular} if $\modulus{\lambda_s} = 1$ for all $s \in S$.

An operator semigroup $(T_s)_{s \in S}$ on a Banach space $E$ is called \emph{bounded} if it satisfies $\sup_{s \in S} \norm{T_s} < \infty$. Note that, as $S$ is a directed set, every operator semigroup $(T_s)_{s \in S}$ is a net, and hence it makes sense to talk about convergence of $(T_s)_{s \in S}$. At this point we recall that, throughout the article, we always endow the operator space $\calL(E)$ with the operator norm, i.e., for us, convergence always means convergence with respect to the operator norm. In the case of a bounded operator semigroup one has the following simple characterization of convergence to the zero operator.

\begin{proposition} \label{prop:bounded-convergence}
	Let $(T_s)_{s \in S}$ be a bounded semigroup of $(S,+)$ on a Banach space $E$. The following assertions are equivalent:
	\begin{enumerate}[\upshape (i)]
		\item $\lim_{s \in S} T_s = 0$.
		\item There exists $s_0 \in S$ such that $\norm{T_{s_0}} < 1$.
		\item $0$ is contained in the closure of the set $\{T_s\colon s \in S\}$.
	\end{enumerate}
\end{proposition}
\begin{proof}
	(i) $\Rightarrow$ (ii): Obvious.
	
	(ii)  $\Rightarrow$ (iii): Let $s_0 \in S$ such that $\norm{T_{s_0}} < 1$. Let $\varepsilon > 0$. Then there exists $n \in \bbN$ such that $\norm{T_{s_0}}^n < \varepsilon$. Hence, 
	\begin{align*}
		\norm{T_{n s_0}} \leq \norm{T_{s_0}}^n < \varepsilon.
	\end{align*}
	Therefore, $0 \in \overline{\{T_s\colon s \in S\}}$.
	
	(iii) $\Rightarrow$ (i): Let $\varepsilon > 0$. Then there exists $s_0 \in S$ such that $\norm{T_{s_0}} \leq \varepsilon$. Thus,
	\begin{align*}
		\norm{T_t} \leq \varepsilon M \qquad \text{for all } t \in s_0 + S,
	\end{align*}
	where $M \coloneqq \sup_{s \in S} \norm{T_s}$. So it follows that $\lim_{s \in S} T_s = 0$. 
\end{proof}

\subsection{The semigroup at infinity}

In \cite[Section 2]{Glueck2019} the concept of the \emph{semigroup at infinity} with respect to the strong operator topology was used to study strong convergence of operator semigroups. In reminiscence of this concept we define the semigroup at infinity now with respect to the operator norm topology.

\begin{definition}
	Let $(T_s)_{s \in S}$ be a semigroup of $(S,+)$ on a Banach space $E$. We call the set
	\begin{align*}
		\sgInftyON{\calT} \coloneqq \bigcap_{r \in S} \overline{\{T_s\colon s \ge r\}}
	\end{align*}
	the \emph{semigroup at infinity} associated with $(T_s)_{s \in S}$ with respect to the operator norm. Since we restrict ourselves to the operator norm topology throughout the paper and since we only consider a single operator semigroup, we will often just call $\sgInftyON{\calT}$ the \emph{semigroup at infinity}.
\end{definition}

Note that the semigroup at infinity consists of all cluster points (with respect to the operator norm) of the net $(T_s)_{s \in S}$.

If the semigroup at infinity, $\sgInftyON{\calT}$, is non-empty and compact, then one can apply the \emph{Jacobs--de Leeuw--Glicksberg} theory to the topological semigroup $\sgInftyON{\calT}$. This yields a smallest non-empty closed ideal $\calI$ in $\sgInftyON{\calT}$ (where \emph{ideal} means that $T \calI \subseteq \calI$ for all $T \in \sgInftyON{\calT}$), and the ideal $\calI$ -- the so-called \emph{Sushkevich kernel} of $\sgInftyON{\calT}$ -- is a compact topological group with respect to operator multiplication. For details we refer for instance to \cite[Section~16.1]{Eisner2015} or to \cite[Theorem~V.2.3]{Engel2000}.

Denote the neutral element in $\calI$ by $P_\infty$ -- it is a projection in $\calL(E)$ which we call the \emph{projection at infinity}; the range of $P_\infty$ is denoted by $E_\infty$. \bigskip

Note that the ``semigroup at infinity'' approach differs from classical applications of JdLG theory to semigroup asymptotics in the following way: classically, one would rather try to apply the JdLG-decomposition to the semigroup 
\begin{align*}
\calT \coloneqq \overline{\{T_s\colon s \in S\}}.
\end{align*}
To make this approach work though, we would need a global compactness requirement of the semigroup $(T_s)_{s \in S}$, in the sense that $\calT$ is compact with respect to the operator norm topology. Generally, this is a far too strong assumption if one is interested in characterising the convergence of $(T_s)_{s \in S}$; this can already be seen by considering the following simple example.

\begin{example}
	Consider the nilpotent right shift $(T_s)_{s \in [0, \infty)}$ on $L^\infty(0, 1)$, i.e.,
	\begin{align*}
	(T_s f)(t) = 
	\begin{cases}
	f(t - s), &\quad \text{if } s < t, \\
	0, &\quad \text{else},
	\end{cases} \qquad (f \in L^\infty(0, 1)). 
	\end{align*}
	Then $T_s$ converges to the zero operator with respect to the operator norm as $s \to \infty$, but $\{T_s\colon s \geq 0\}$ is not even relatively compact in the strong operator topology.

	If we replace $L^\infty(0,1)$ with $L^p(0,1)$ for $p \in [1,\infty)$, the set $\{T_s\colon s \geq 0\}$ becomes relatively compact with respect to the strong operator topology, but it is still not relatively compact with respect to the operator norm topology.
\end{example}

To overcome this obstacle, in the following theorem we will apply the JdLG-decomposition to the semigroup at infinity. This result is very close in spirit to a similar theorem for the strong operator topology that can be found in \cite[Theorem~2.2]{Glueck2019}. 

\begin{theorem} \label{thm:JdLG-semigroup-infinity}
	Let $(T_s)_{s \in S}$ be a bounded semigroup of $(S,+)$ on a Banach space $E$ and assume that the semigroup at infinity, $\sgInftyON{\calT}$, is non-empty and compact. Set $\calT \coloneqq \overline{\{T_s\colon s \in S\}} \subseteq \calL(E)$. Then the following assertions hold:
	\begin{enumerate}[\upshape (a)]
		\item The projection at infinity, $P_\infty$, commutes with all operators in $\calT$, and $\calT P_\infty = \sgInftyON{\calT} P_\infty$.
		
		\item The semigroup at infinity, $\sgInftyON{\calT}$, is a group with respect to operator multiplication with neutral element $P_\infty$. Moreover, we have
			\begin{align*}
				\calT|_{E_\infty} = \sgInftyON{\calT}|_{E_\infty} = \overline{\{T_s \colon s \in S\}|_{E_\infty}}^{\calL(E_\infty)},
			\end{align*}
			and this set is a compact subgroup of the bijective operators in $\calL(E_\infty)$. Finally, $\sgInftyON{\calT}$ and $\sgInftyON{\calT}|_{E_\infty}$ are isomorphic (in the category of topological groups) via the mapping $R \mapsto R|_{E_\infty}$.
		
		\item We have $\lim_{s \in S} T_s|_{\ker P_\infty} = 0$ with respect to the operator norm on $\calL(\ker P_\infty)$. 
		
		\item For every vector $x \in E$ the following assertions are equivalent:
			\begin{enumerate}[\upshape (i)]
				\item $P_\infty x = 0$.
				\item $0$ is contained in the weak closure of the orbit $\{T_s x \mid \, s \in S\}$.
				\item The net $(T_s x)_{s \in S}$ norm converges to $0$ in $E$.
				\item We have $Rx = 0$ for each $R \in \sgInftyON{\calT}$.
				\item We have $Rx = 0$ for at least one $R \in \sgInftyON{\calT}$.
			\end{enumerate}
		\item If the underlying scalar field of $E$ is complex, then the semigroup $(T_s)_{s \in S}$ has discrete spectrum, i.e, 
		\begin{align*}
		E_\infty = \overline{\lin} \{x \in E\colon \forall \, s \in S \ \exists\, \lambda_s \in \bbT \text{ with } T_s x = \lambda_s x\}.
		\end{align*}		 
	\end{enumerate}
\end{theorem}

Note that the first part of assertion~(a) implies that every operator in $\calT$ -- and thus in particular every operator $T_s$ -- leaves $E_\infty$ and $\ker P_\infty$ invariant.

\begin{proof}[Proof of Theorem~\ref{thm:JdLG-semigroup-infinity}]
	(a) The first assertion is clear since $\calT$ is commutative. Moreover, we have $\calT \sgInftyON{\calT} \subseteq \sgInftyON{\calT} \subseteq \calT$, where the second inclusion is obvious and the first inclusion follows easily from the definitions of $\calT$ and $\calT_\infty$. Therefore,
	\begin{align*}
		\calT P_\infty = \calT P_\infty P_\infty \subseteq \sgInftyON{\calT} P_\infty \subseteq \calT P_\infty.
	\end{align*}
	
	(c) Since $P_\infty$ is trivial on $\ker P_\infty$, we have $0 \in \overline{\{T_s|_{\ker P_\infty} \colon s \in S\}}$; this is equivalent to $\lim_{s \in S} T_s|_{\ker P_\infty} = 0$ by Proposition~\ref{prop:bounded-convergence}. 
	
	(b) Let $\calI \subseteq \sgInftyON{\calT}$ denote the Sushkevich kernel of $\sgInftyON{\calT}$, i.e., the smallest non-empty closed ideal in the semigroup $\sgInftyON{\calT}$ (see the discussion before the theorem). We show that $\sgInftyON{\calT} = \calI$. To this end, let $R \in \sgInftyON{\calT}$. Then $R$ is a cluster point of the net $(T_s)_{s \in S}$, so there exists a subnet $(T_{s_j})_j$ that converges to $R$. It follows from assertion~(c), which we have already proved, that $T_{s_j}(\id_E - P_\infty) \to 0$, so $R(\id_E - P_\infty) = 0$ and hence, $R = RP_\infty$. Since $P_\infty \in \calI$ and since $\calI$ is an ideal in $\sgInftyON{\calT}$ we conclude that $R \in \calI$. We have thus proved that $\sgInftyON{\calT}$ is a group with respect to operator multiplication and that its neutral element is $P_\infty$.
		
	Next we show the equalities in the displayed formula. One has $\calT|_{E_\infty} = \sgInftyON{\calT}|_{E_\infty}$ by (a). As the restriction map from $\calL(E)$ to $\calL(E_\infty; E)$ is continuous, we have $\calT|_{E_\infty} \subseteq \overline{\{T_s \colon s \in S\}|_{E_\infty}}$. The converse inclusion follows from $\overline{\{T_s \colon s \in S\}|_{E_\infty}} P_\infty \subseteq \calT$.
	
	Since $\sgInftyON{\calT}$ is a group with neutral element $P_\infty$, it readily follows that $\sgInftyON{\calT}|_{E_\infty}$ is a subgroup of the bijective operators in $\calL(E_\infty)$. The mapping
	\begin{align*}
		\sgInftyON{\calT} \ni R \mapsto R|_{E_\infty} \in \sgInftyON{\calT}|_{E_\infty}
	\end{align*}
	is clearly a surjective and continuous group homomorphism and consequently, $\sgInftyON{\calT}|_{E_\infty}$ is compact. If $R|_{E_\infty} = \id_{E_\infty}$ for some $R\in \sgInftyON{\calT}$, then $P_\infty = RP_\infty = R$, so our group homomorphism is also injective. Finally, it is also a homeomorphism by the compactness of its domain and range.
	
	(d) Fix $x \in E$. 
	
	(iv) $\Rightarrow$ (i) $\Rightarrow$ (v) $\Rightarrow$ (iii) $\Rightarrow$ (iv):
	Clearly, since $P_\infty \in \sgInftyON{\calT}$, (iv) implies (i) and (i) implies (v). Furthermore, (v) implies $0 \in \overline{ \{T_s x \mid s \in S\} }$ which is equivalent to $\lim_{s \in S} T_s x = 0$, i.e., (iii), due to the boundedness of the semigroup. Moreover, if (iii) holds and $\varepsilon > 0$ is fixed, then there exists $s \in S$ such that $\{T_t x \mid t \geq s\} \subseteq \varepsilon \Ball$, where $\Ball$ denotes the closed unit ball in $E$. Thus, $\sgInftyON{\calT} x \subseteq \varepsilon\Ball$. Since $\varepsilon > 0$ was arbitrary, it follows that $\sgInftyON{\calT} x = \{0\}$, i.e., (iv) holds. 
	
	(ii) $\Leftrightarrow$ (iii):
	Obviously, (iii) implies (ii). Conversely, suppose that (ii) holds. Then it follows that $0$ is contained in the weak closure of the set $\{T_s P_\infty x \mid s \in S\}$. Moreover, it follows from~(a) that the set $\{T_s P_\infty \mid s \in S\}$ is a subset of $\sgInftyON{\calT} P_\infty$ and thus relatively compact in $\calL(E)$. Hence, $\{T_s P_\infty x \mid s \in S\}$ is relatively strongly compact and thus its closure coincides with its weak closure. Hence, $0$ is contained in the strong closure of $\{T_s P_\infty x \mid s \in S\}$, so $T_s P_\infty x \to 0$ due to the boundedness of the semigroup. If we apply the implication from~(iii) to~(i), which we have already shown, to the vector $P_\infty x$, this yields $P_\infty x = P_\infty (P_\infty x) = 0$.
	
	(e) Recall that, by (b), $\calG \coloneqq \overline{\{T_s \colon s \in S\}|_{E_\infty}} \subseteq \calL(E_\infty)$ is a compact group with respect to the operator norm on $\calL(E_\infty)$. Let $\calG^*$ denote the dual group of $\calG$. According to \cite[Corollary~15.18]{Eisner2015} we have
	\begin{align*}
		E_\infty = \, & \overline{\lin} \{x \in E_\infty\colon \exists \, \xi \in \calG^* \; \forall \, R \in \calG\colon Rx = \xi(R)x\} \\
		\subseteq \, & \overline{\lin} \{x \in E_\infty\colon \forall \, s \in S \ \exists\, \lambda_s \in \bbT\colon T_s x = \lambda_s x\} \subseteq E_\infty.
	\end{align*}
	Now let $x \in E$ be an eigenvector associated to the unimodular eigenvalue $\lambda = (\lambda_s)_{s \in S}$. Consider $y \coloneqq (I - P_\infty) x \in \ker P_\infty$. Then $T_s y \to 0$ and $T_s y = \lambda_s y$ for each $s \in S$. Since $\modulus{\lambda_s} = 1$ for all $s \in S$, this implies $y = 0$, i.e., $x \in E_\infty$.
\end{proof}

\begin{remark} \label{rem:sg-at-infty-is-also-a-group-in-the-strong-case}
	\begin{enumerate}[\upshape (a)]
		\item For the strong operator topology, the analogue result to Theorem~\ref{thm:JdLG-semigroup-infinity} is \cite[Theorem~2.2]{Glueck2019}. The assertion that the semigroup at infinity is automatically a group in case that it is non-empty and compact is not included in this reference, but it is also true in the situation there; this can be shown by exactly the same argument as in our proof of Theorem~\ref{thm:JdLG-semigroup-infinity}(b). This shows that the semigroup at infinity is minimal in the sense that there is no smaller topological group that contains all the information about the asymptotic behaviour of the semigroup.

		\item For the strong operator topology, the statement in Theorem~\ref{thm:JdLG-semigroup-infinity}(e) holds, too, although that was not observed in \cite[Theorem~2.2]{Glueck2019}. 
	\end{enumerate}
\end{remark}

As a consequence of the above theorem, operator norm convergence of a semigroup can be characterised in terms of its semigroup at infinity. Let us state this explicitly in the following corollary.

\begin{corollary} \label{cor:characterization-of-sg-convergence}
	For every bounded semigroup $(T_s)_{s \in S}$ on a Banach space $E$ the following assertions are equivalent:
	\begin{enumerate}[\upshape (i)]
		\item $(T_s)_{s \in S}$ converges (with respect to the operator norm).
		\item $\sgInftyON{\calT}$ is a singleton. 
		\item $\sgInftyON{\calT}$ is non-empty and compact, and acts as the identity on $E_\infty$.
		\item $\sgInftyON{\calT}$ is non-empty and compact, and $(T_s)_{s \in S}$ acts as the identity on $E_\infty$.
	\end{enumerate}
	If the equivalent conditions~\textup{(i)--(iv)} are satisfied, then $\lim_{s \in S}T_s$ equals $P_\infty$, the projection at infinity.
	
	If the underlying scalar field of $E$ is complex, the above assertions~ \textup{(i)--(iv)} are also equivalent to:
	\begin{enumerate}[\upshape (v)]
		\item $\sgInftyON{\calT}$ is non-empty and compact, and $\one \coloneqq (1)_{s \in S}$ is the only unimodular eigenvalue of $(T_s)_{s \in S}$.
	\end{enumerate}
\end{corollary}
\begin{proof}
	(i) $\Rightarrow$ (ii): If the net $(T_s)_{s \in S}$ converges, then its limit is the only cluster point of $(T_s)_{s \in S}$. Hence, $\sgInftyON{\calT}$ is a singleton.
	
	(ii) $\Rightarrow$ (iii): Assertion~(ii) implies $\sgInftyON{\calT} = \{P_\infty\}$, and $P_\infty$ acts trivially on $E_\infty$.
	
	(iii) $\Rightarrow$ (iv): By Theorem~\ref{thm:JdLG-semigroup-infinity}(a) we have $\calT|_{E_\infty} = \sgInftyON{\calT}|_{E_\infty}$, so~(iii) implies~(iv).
	
	(iv) $\Rightarrow$ (i): By Theorem~\ref{thm:JdLG-semigroup-infinity}(c), assertion~(iv) implies that $\lim_{s \in S} T_s = P_\infty$.
	
	(iv) $\Leftrightarrow$ (v): By Theorem~\ref{thm:JdLG-semigroup-infinity}(e), $(T_s)_{s \in S}$ acts as the identity on $E_\infty$ if and only if $\one \coloneqq (1)_{s \in S}$ is the only unimodular eigenvalue of $(T_s)_{s \in S}$.
\end{proof}

\begin{remark} \label{rem:embedded-semigroup-and-strong-op-topology}
	We note once again that our results in this subsection, as well as their proofs, are quite close to similar results for the strong operator topology from \cite[Subsection~2.2]{Glueck2019}. The relation between the semigroups at infinity with respect to the operator norm topology and with respect to the strong operator topology can also be formalised in the following sense.
	
	If $(T_s)_{s \in S}$ is an operator semigroup on a Banach space $E$ one can, for each $s \in S$, define an operator $R_s$ on the Banach space $\calL(E)$ by
	\begin{align*}
		R_s\colon \calL(E) \to \calL(E), \quad A \mapsto T_s A.
	\end{align*}
	Then $(R_s)_{s \in S}$ is a bounded semigroup on the Banach space $\calL(E)$, and topological properties of $(R_s)_{s \in S}$ with respect to the strong operator topology translate into topological properties of $(T_s)_{s \in S}$ with respect to the operator norm. This observation can be used as a basis to derive the theory of the semigroup at infinity with respect to the operator norm from the corresponding theory with respect to the strong topology presented in \cite{Glueck2019}.
	
	However, in the present section we prefer to give more direct proofs in order to make our work more self-contained and to improve its accessibility for readers not familiar with \cite{Glueck2019}.
\end{remark}

In order to apply Theorem~\ref{thm:JdLG-semigroup-infinity} and Corollary~\ref{cor:characterization-of-sg-convergence} one needs criteria to ensure that the semigroup at infinity is non-empty and compact; in a general setting, such criteria can be found in the following proposition.

\begin{proposition} \label{prop:characterisation-of-compact-and-non-empty}
	For every bounded semigroup $(T_s)_{s \in S}$ on a Banach space $E$, the following assertions are equivalent:
	\begin{enumerate}[\upshape (i)]
		\item The semigroup at infinity is non-empty and compact.
		\item Every subnet of $(T_s)_{s\in S}$ has a convergent subnet.
		\item Every universal subnet of $(T_s)_{s \in S}$ converges.
	\end{enumerate}
	In case that $S$ contains a cofinal sequence, the above assertions~\textup{(i)--(iii)} are also equivalent to:
	\begin{enumerate}[\upshape (iv)]
		\item For every cofinal sequence $(s_n)_{n \in \bbN}$ in $S$, the sequence $(T_{s_n})_{n \in \bbN}$ has a convergent subsequence.
	\end{enumerate}
\end{proposition}
\begin{proof}
	(i) $\Leftarrow$ (ii) $\Leftrightarrow$ (iii): These implications follow from general topological properties; see Lemma~\ref{lemma:set-of-cluster-points}.
	
	(i) $\Rightarrow$ (ii): Note that one has $\lim_{s \in S} (T_s(I - P_\infty)) = 0$ by Theorem~\ref{thm:JdLG-semigroup-infinity}(c). Moreover, the net $(T_s P_\infty)_{s \in S}$ is contained in the compact set $\sgInftyON{\calT}P_\infty$ by Theorem~\ref{thm:JdLG-semigroup-infinity}(a). Thus each of it subnets has a convergent subnet. Since
	\begin{align*}
		T_s = T_s P_\infty + T_s (I - P_\infty) \qquad \text{for all } s \in S,
	\end{align*}
	this shows that every subnet of $(T_s)_{s \in S}$ has a convergent subnet.
	
	Now assume that $S$ contains a co-final subsequence.
	
	(iii) $\Rightarrow$ (iv) $\Rightarrow$ (i): This, again, follows from the general Lemma~\ref{lemma:set-of-cluster-points}.
\end{proof}

If $(x_\alpha)_\alpha$ is a net in an arbitrary metric (or topological) space whose set of cluster points is non-empty and compact, then the set of cluster points of a fixed subnet of $(x_\alpha)_\alpha$ might well be empty. The implication (i) $\Rightarrow$ (ii) in Proposition~\ref{prop:characterisation-of-compact-and-non-empty} show that the situation is different for our semigroup setting. A nice consequence of this observation is the subsequent Corollary~\ref{cor:subsemigroup-non-empty-and-compact}. For a proper understanding of that corollary, the following algebraic observation is important.

\begin{remark} \label{remark:order-on-subsemigroups}
	Let $R$ be a subsemigroup of $S$ that contains $0$. Denote the pre-order on $R$ inherited from $S$ by $\le_S$ and denote the pre-order on $R$ induced by its semigroup operation by $\le_R$. For all $r_1,r_2 \in R$ one then has the implication
	\begin{align*}
		r_1 \le_R r_2 \quad \Longrightarrow \quad r_1 \le_S r_2.
	\end{align*}
	Note that $\le_R$ and $\le_S$ do not coincide in general, which can be seen, for instance, by considering the subsemigroup $\{0\} \cup [1,\infty)$ of $([0,\infty),+)$.
	
	Now, let $X$ be a set and for each $r \in R$, let $x_r \in X$. Let us use, within this remark, the notations $(x_r)_{r \in (R,\le_R)}$ and $(x_r)_{r \in (R,\le_S)}$ to distinguish the nets that we obtain be considering the different pre-orders $\le_R$ and $\le_S$ on $R$. Then it follows from the implication above that the net $(x_r)_{r \in (R,\le_R)}$ is a subnet of $(x_r)_{r \in (R,\le_S)}$.
	
	In particular, if $R$ is cofinal in $S$ and $(x_s)_{s \in S}$ is a net in $X$, then $(x_r)_{r \in (R,\le_R)}$ is a subnet of $(x_s)_{s \in S}$.
\end{remark}

\begin{corollary} \label{cor:subsemigroup-non-empty-and-compact}
	Let $E$ be a Banach space. Let $R$ be a subsemigroup of $S$ that contains $0$ and is cofinal in $S$ and let $(T_s)_{s \in S}$ is a bounded semigroup on $E$ whose associated semigroup at infinity is non-empty and compact.
	
	Then the semigroup at infinity associated with $(T_s)_{s \in R}$ is also non-empty and compact, and the projections at infinity of $(T_s)_{s \in S}$ and $(T_s)_{s \in R}$ coincide.
\end{corollary}

Note that in the corollary the semigroup $R$ is endowed with the order induced by its semigroup operation (denoted by $\le_R$ in Remark~\ref{remark:order-on-subsemigroups}). For any other order on $R$ (for instance the order inherited from $S$) we did not even define the notion \emph{semigroup at infinity}.

\begin{proof}[Proof of Corollary~\ref{cor:subsemigroup-non-empty-and-compact}]
	It follows from Remark~\ref{remark:order-on-subsemigroups} that $(T_s)_{s \in R}$ is a subnet of $(T_s)_{s \in S}$. In particular, every universal subnet of $(T_s)_{s \in R}$ is also a universal subnet of $(T_s)_{s \in S}$ and thus convergent by Proposition~\ref{prop:characterisation-of-compact-and-non-empty}. Hence, by the same proposition the semigroup at infinity associated with $(T_s)_{s \in R}$ is non-empty and compact.
	
	Let $P_\infty$ and $Q_\infty$ denote the projections at infinity of $(T_s)_{s \in S}$ and $(T_s)_{s \in R}$, respectively. Those two projections commute. It follows from Theorem~\ref{thm:JdLG-semigroup-infinity}(c) that $\lim_{s \in S} T_s|_{\ker P_\infty} = 0$ and thus, in particular, $\lim_{s \in R} T_s|_{\ker P_\infty} = 0$; Theorem~\ref{thm:JdLG-semigroup-infinity}(d), applied to the semigroup $(T_s)_{s \in R}$, thus implies that $Q_\infty x = 0$ for every $x \in \ker P_\infty$, i.e., $\ker P_\infty \subseteq \ker Q_\infty$.
	
	Conversely, it also follows from Theorem~\ref{thm:JdLG-semigroup-infinity}(c) that $\lim_{s \in R} T_s|_{\ker Q_\infty} = 0$, so Proposition~\ref{prop:bounded-convergence} implies that even $\lim_{s \in S} T_s|_{\ker Q_\infty} = 0$. Theorem~\ref{thm:JdLG-semigroup-infinity}(d), applied to the semigroup $(T_s)_{s \in S}$, thus implies that $P_\infty x = 0$ for every $x \in \ker Q_\infty$, i.e., $\ker Q_\infty \subseteq \ker P_\infty$. Therefore, we proved that the commuting projections $P_\infty$ and $Q_\infty$ have the same kernel. The general observation that two commuting projections coincide if their kernels coincide, thus yields $P_\infty = Q_\infty$.
\end{proof}

In order to determine the projection $P_\infty$ in concrete situations the following proposition is quite useful; it shows that $P_\infty$ is uniquely determined by some of its properties listed in Theorem~\ref{thm:JdLG-semigroup-infinity}.

\begin{proposition} \label{prop:uniqueness-of-proposition}
	Let $(T_s)_{s \in S}$ be a bounded semigroup on a Banach space $E$ and let $P \in \calL(E)$ be a projection that commutes with all operators $T_s$. Consider the following assertions:
	\begin{enumerate}[\upshape (a)]
		\item $\lim_s T_s|_{\ker P} = 0$ (with respect to the operator norm on $\calL(\ker P)$).
				
		\item The set $\{T_s|_{PE} \colon s \in S\}$ is relatively compact in $\calL(PE)$.

		\item The net $(T_s x)_{s \in S}$ does not converge to $0$ for any $x \in PE \setminus \{0\}$.
	\end{enumerate}
	If assertions~(a) and~(b) are satisfied, then the semigroup at infinity, $\sgInftyON{\calT}$, is non-empty and compact, and the projection at infinity satisfies
	\begin{align*}
		P_\infty E \subseteq PE \qquad \text{and} \qquad \ker P_\infty \supseteq \ker P.
	\end{align*}
	If all assertions assertions~\textup{(a)--(c)} are satisfied, then in addition $P_\infty = P$.
\end{proposition}
\begin{proof}
	First note that the semigroup leaves both the kernel and the range of $P$ invariant since $P$ commutes with each operator $T_s$. Now assume that~(a) and~(b) are satisfied and let $(T_{s_j})$ be a universal subnet of $(T_s)_{s \in S}$. By~(a), $(T_{s_j}|_{\ker P})$ converges to $0$ and by~(b), $(T_{s_j}|_{PE})$ is convergent. Thus, the net $(T_{s_j})$ is convergent, which proves that $\sgInftyON{\calT}$ is non-empty and compact by Proposition~\ref{prop:characterisation-of-compact-and-non-empty}. It follows from assumption~(a) and Theorem~\ref{thm:JdLG-semigroup-infinity}(d) that $\ker P_\infty \supseteq \ker P$. To show that $P_\infty E \subseteq PE$, let $x \in P_\infty E$. We have $(\id_E-P)x \in \ker P \subseteq \ker P_\infty$, and since $P_\infty$ and $P$ commute, this implies that $0 = (\id_E-P)P_\infty x  = (\id_E-P)x$, so $x = Px \in PE$.
	
	Now assume in addition that assumption~(c) is satisfied. We show that the inclusion $\ker P_\infty \subseteq \ker P$ is also satisfied then. Let $x \in \ker P_\infty$. Since $P$ and $P_\infty$ commute, the projection $P$ leaves $\ker P_\infty$ invariant, i.e., we also have $Px \in \ker P_\infty$. Hence, $T_s P x \to 0$ by Theorem~\ref{thm:JdLG-semigroup-infinity}(d), so it follows from assumption~(c) that $Px = 0$. We thus proved that the kernels of $P_\infty$ and $P$ coincide, so $P_\infty = P$.
\end{proof}

\subsection{Powers of a single operator}

In this subsection we consider time-discrete semigroups, i.e., semigroups of the form $(T^n)_{n \in \bbN_0}$ for a single operator $T$. Let us first note in the following lemma that, in this case, the semigroup at infinity is non-empty and compact if and only if the entire set $\{T^n\colon n \in \bbN_0\}$ is relatively compact in $\calL(E)$.

\begin{lemma} \label{lemma:compactness-discrete-case}
	Let $T \in \calL(E)$ be a power-bounded operator on a Banach space $E$. Then the semigroup at infinity, $\sgInftyON{\calT}$, associated to the semigroup $(T^n)_{n \in \bbN_0}$ is non-empty and compact if and only if the set $\calT = \{T^n\colon  n \in \bbN_0\}$ is relatively compact in $\calL(E)$.
\end{lemma}
\begin{proof}
	``$\Rightarrow$''
	Let $(T^{n_k})_{k \in \bbN}$ be an arbitrary sequence in $\calT$; we have to distinguish two cases since this sequence might not be a subsequence of $(T^n)_{n \in \bbN_0}$. In the first case, the index sequence $(n_k)_{k \in \bbN}$ is bounded; then, by the pigeon hole principle, it has a constant subsequence, so $(T^{n_k})_{k \in \bbN}$ has a constant, thus convergent, subsequence.
	
	In the second case the index sequence $(n_k)_{k \in \bbN}$ is unbounded. Then it has a subsequence $(n_{k_j})_{j \in \bbN}$ that is cofinal in $\bbN_0$. Hence, Proposition~\ref{prop:characterisation-of-compact-and-non-empty} yields that $(T^{n_{k_j}})_{j \in \bbN}$ has a convergent subsequence, and the latter is also a subsequence of $(T^{n_k})_{k \in \bbN}$. 
	
	``$\Leftarrow$'' The implication follows directly from Proposition~\ref{prop:characterisation-of-compact-and-non-empty}. 
\end{proof}

Now we derive a spectral characterization of the compactness and non-emptiness of the semigroup at infinity associated to a single operator. 

\begin{proposition} \label{prop:compactness-for-single-operators}
	Let $T \in \calL(E)$ be a power-bounded operator on a complex Banach space $E$ and consider the semigroup $(T^n)_{n \in \bbN_0}$ on $E$. Then the following two assertions are equivalent:
	\begin{enumerate}[\upshape (i)]
		\item The semigroup at infinity, $\sgInftyON{\calT}$, is non-empty and compact.
		\item All spectral values of $T$ on the unit circle are poles of the resolvent of $T$.
	\end{enumerate}
	In this case, $P_\infty$ coincides with the spectral projection of $T$ associated with $\spec(T) \cap \bbT$.
\end{proposition}
\begin{proof}
	(i) $\Rightarrow$ (ii): Let $\lambda \in \bbT$ be a spectral value of $T$. Let $\calK$ denote the closed convex hull of the relatively compact set 
	\begin{align*}
	\bbT \cdot \{T^n\colon  n \in \bbN_0\};
	\end{align*}
	then $\calK$ is compact, too. Moreover, the operator $(r\lambda - \lambda) \Res(r\lambda,T)$ is contained in $\calK$ for each $r > 1$; this is a consequence of the Neumann series representation of the resolvent. Consequently, the net $\big((r\lambda - \lambda) \Res(r\lambda,T)\big)_{r \in (1,\infty)}$ (where $(1,\infty)$ is directed conversely to the order inherited from $\bbR$) has a convergent subnet. This shows, according to Proposition~\ref{prop:pole-of-resolvent-by-resolvent-convergence} in the appendix, that $\lambda$ is a pole of $\Res(\argument,T)$.
		
	(ii) $\Rightarrow$ (i): Note that, as a consequence of~(ii), $\spec(T) \cap \bbT$ is isolated from the rest of the spectrum of $T$; let $P$ denote the spectral projection associated with $\spec(T) \cap \bbT$. We show that $P$ satisfies the assumptions (a)--(c) in Proposition~\ref{prop:uniqueness-of-proposition}.
	
	The spectral radius of $T|_{\ker P}$ is strictly less than $1$, so $T|_{\ker P}^n \to 0$ as $n \to \infty$; this proves assumption~(a). In order to show assumptions~(b) and~(c), note that the set $\spec(T) \cap \bbT$ is finite as a consequence of~(ii), and enumerate its elements (if any exist) as $\lambda_1,\dots,\lambda_m$.
	
	By assumption, each $\lambda_k$ is a pole of the resolvent of $T$, and its pole order equals $1$ since $T$ is power bounded. Hence, $T$ acts as $\lambda_k$ times the identity on the range of the associated spectral projection $P_k$. It follows that $T$ acts on $PE = P_1E \oplus \dots \oplus P_mE$ as the multiplication with the tuple $(\lambda_1, \dots,\lambda_m)$, which readily implies that $\{(T|_{PE})^n \colon n \in \bbN_0\}$ is relatively compact with respect to the operator norm and that $T^nx$ does not converge to $0$ as $n \to \infty$ for any $x \in PE$. Thus, all assumptions~(a)--(c) of Proposition~\ref{prop:uniqueness-of-proposition} are satisfied, which shows that $\sgInftyON{\calT}$ is non-empty and compact and $P = P_\infty$.
\end{proof}

\subsection{Semigroups which contain a quasi-compact operator}

Recall that a bounded operator $T$ on a Banach space $E$ is called \emph{quasi-compact} if there exists a compact operator $K$ on $E$ and $n \in \bbN$ such that $\norm{T^n - K} < 1$. It is well known that, if the underlying scalar field is complex, a quasi-compact operator $T$ has at most finitely many spectral values on the complex unit circle, and that all those spectral values are poles of the resolvent of $T$ with finite-rank residuum. Hence, the spectral projection associated to the part of the spectrum on the unit circle has finite rank.

Quasi-compact operators -- and in particular, of course, compact operators -- appear quite often in concrete applications. This is why the following proposition, in conjunction with Theorem~\ref{thm:JdLG-semigroup-infinity} and Corollary~\ref{cor:characterization-of-sg-convergence}, is very useful.

\begin{proposition} \label{prop:quasi-compact}
	Let $(T_s)_{s \in S}$ be a bounded semigroup on a Banach space $E$ such that, for some $s_0 \in S$, the operator $T_{s_0}$ is quasi-compact. Then the semigroup at infinity associated to $(T_s)_{s \in S}$ is non-empty and compact, and the projection at infinity has finite rank. 
\end{proposition}
\begin{proof}
	We may assume that the underlying scalar field of $E$ is complex, since otherwise we can consider a complexification of $E$. According to Proposition~\ref{prop:compactness-for-single-operators} the semigroup at infinity associated to $(T_{s_0}^n)_{n \in \bbN_0}$ is non-empty and compact; let $P$ denote the projection at infinity associated to this semigroup at infinity.
	
	Then $P$ commutes with each operator $T_s$, so both $\ker P$ and $PE$ are invariant under the action of the semigroup $(T_s)_{s \in [0,\infty)}$. Moreover, $(T_{s_0}|_{\ker P})^n \to 0$ as $n \to \infty$, so it follows from Proposition~\ref{prop:bounded-convergence} that actually $\lim_{s \in S}T_s|_{\ker P} = 0$. Additionally, it follows from Proposition~\ref{prop:compactness-for-single-operators} and the quasi-compactness of $T_{s_0}$ that $PE$ is finite-dimensional. Since our semigroup is bounded, the set $\{T_s|_{PE}\colon s \in S\}$ is thus relatively compact in $\calL(PE)$, so it follows from Proposition~\ref{prop:uniqueness-of-proposition} that the semigroup at infinity associated with $(T_s)_{s \in S}$ is non-empty and compact, and that the projection at infinity, $P_\infty$, satisfies $P_\infty E \subseteq PE$. Hence, $P_\infty$ has finite rank.
\end{proof}

In the situation of Proposition~\ref{prop:quasi-compact}, the projections at infinity associated with $(T_s)_{s \in S}$ and with $(T_{s_0}^n)_{n \in \bbN_0}$ coincide if the subsemigroup $\{ns_0\colon n \in \bbN_0\}$ is cofinal in $S$ (see Corollary~\ref{cor:subsemigroup-non-empty-and-compact}). Without this additional assumption, the projections at infinity do not need to coincide, as the following examples show.

\begin{examples}
	\begin{enumerate}[(a)]
		 \item Consider the semigroup $S = [0,\infty)$ where the semigroup operation is given by the maximum operator $\lor$; then the order in $S$ coincides with the usual order on $[0,\infty)$. Now, let $E = \bbC^2$, let $Q \in \calL(\bbC^2)$ be the projection onto the first component and define
		\begin{align*}
			T_s =
			\begin{cases}
				\id_{\bbC^2} \quad & \text{if } s \in [0,1], \\
				Q \quad & \text{if } s \in (1,\infty).
			\end{cases}
		\end{align*}
		Then $(T_s)_{s \in S}$ is a bounded semigroup of $([0,\infty),\lor)$, its semigroup at infinity is non-empty and compact and its projection at infinity equals $Q$. The operator $T_1$ is compact, but the projection at infinity associated to $(T_1^n)_{n \in \bbN_0}$ is $\id_{\bbC^2}$.
		
		\item Here is also an example where the underlying semigroup is cancellative: Let $S = [0,\infty)^2$, together with the componentwise addition $+$. Let $E = \mathbb{C}^2$, let $Q \in \calL(\bbC^2)$ denote the projection onto the first component and $P \in \calL(\bbC^2)$ the projection onto the second component. We define a representation $(T_{(s,t)})_{(s,t) \in [0,\infty)^2}$ by
		\begin{align*}
			T_{(s,t)} =
			\begin{cases}
				\id_{\bbC^2} \quad & \text{if } s = 0 \text{ and } t = 0, \\
				Q \quad & \text{if } s > 0 \text{ and } t = 0, \\
				P \quad & \text{if } s = 0 \text{ and } t > 0, \\
				0 \quad & \text{if } s > 0 \text{ and } t > 0.
			\end{cases}
		\end{align*}
		Then $(T_{(s,t)})_{(s,t) \in [0,\infty)^2}$ is a bounded semigroup with non-empty and compact semigroup at infinity; its projection at infinity equals $0$. The operator $T_{(0,1)}$ is compact, but the projection at infinity associated with $(T_{(0,1)}^n)_{n \in \bbN_0}$ equals $P$.
	\end{enumerate}
\end{examples}

\subsection{Beyond the quasi-compact case}

While the situation of Proposition~\ref{prop:quasi-compact} is most important for applications, it is not completely satisfying from a theoretical point of view. Indeed, for every Banach space $E$ and every commutative monoid $(S,+)$ the semigroup at infinity associated to the trivial semigroup $(\id_E)_{s \in S}$ is non-empty and compact, but $\id_E$ is not quasi-compact unless $E$ is finite-dimensional.

In the case of a time-discrete semigroup $(T^n)_{n \in \bbN_0}$ the non-quasi-compact case is still covered by Proposition~\ref{prop:compactness-for-single-operators} -- where non-quasi-compactness of $T$ means precisely that at least one spectral value on the unit circle has infinite-dimensional eigenspace. It would be satisfying to have a similar result for more general semigroups $(S,+)$ at hand, at least for the semigroup $([0,\infty),+)$. However, the following example shows the things are not that simple.

\begin{example} \label{ex:easy-spectrum-not-sufficient-for-semigroups-with-real-times}
	There exists an $L^2$-space and a bounded positive semigroup $\calT = (T_t)_{t \in [0,\infty)}$ on it with the following properties:
	\begin{enumerate}[\upshape(a)]
		\item The spectrum of every operator $T_t$ is finite and consists of poles of the resolvent.
		\item The semigroup at infinity, $\sgInftyON{\calT}$, is not compact.
	\end{enumerate}
	Indeed, let $U \subseteq \bbT$ denote the group of all roots of unity and consider the space $\ell^2(U)$. Note that there exists a group homomorphism $\varphi\colon \bbR \to \bbQ$ which acts as the identity on $\bbQ$ (the existence of $\varphi$ follows from the fact the $\bbR$, seen as a vector space over $\bbQ$, possesses a basis that contains the number $1$). We define $\calT$ by
	\begin{align*}
		T_tf(z) = f(\ue^{2\pi \ui \varphi(t)}z)
	\end{align*}
	for $t \in [0,\infty)$, $f \in \ell^2(U)$ and $z \in U$. Obviously, the semigroup obtained this way is bounded and positive.
	
	For every time $t$ there exists an integer $n \in \bbN$ such that $T_t^n = I$ (indeed, one simply has to choose $n$ such that $n\varphi(t)$ is an integer). Hence, every operator $T_t$ is algebraic (i.e., mapped to $0$ by a polynomial), so it follows that property~(a) is satisfied.
	
	On the other hand, choose a sequence $(q_n)_{n \in \bbN}$ of positive rational numbers which converges to $\infty$ and such that $\ue^{2\pi \ui q_n} \not= \ue^{2\pi \ui q_m}$ whenever $n \not= m$. By applying the sequence $(T_{q_n})_{n \in \bbN}$ to any canonical unit vector in $\ell^2(U)$ we can see that no subsequence of this sequence converges (not even strongly) as $n \to \infty$. Hence, it follows from Proposition~\ref{prop:characterisation-of-compact-and-non-empty} that the semigroup at infinity is either empty or not compact. Since $\sgInftyON{\calT}$ clearly contains the identity operator, we thus conclude that $\sgInftyON{\calT}$ is not compact.
\end{example}

\begin{remark} \label{rem:from-subsemigroups-to-semigroup-is-not-always-possible}
	\begin{enumerate}[(a)]
		\item In the situation of Example~\ref{ex:easy-spectrum-not-sufficient-for-semigroups-with-real-times} the semigroup at infinity associated with the time discrete semigroup $(T_{nt})_{n \in \bbN}$ is, for any time $t \in (0,\infty)$, non-empty and compact; this follows from Proposition~\ref{prop:compactness-for-single-operators}. On the other hand, the semigroup at infinity associated with the entire semigroup $\calT$ is not compact. This shows that the implication in Corollary~\ref{cor:subsemigroup-non-empty-and-compact} does not have a simple converse.
	
		\item It is easy to modify Example~\ref{ex:easy-spectrum-not-sufficient-for-semigroups-with-real-times} in such a way that all orbits of the semigroup become relatively compact: just replace $\ell^2(U)$ with $L^2(\bbT)$ in the example and construct the semigroup in the same way. Then, for each $f \in L^2(\bbT)$, the orbit $\{T_tf\colon t \in [0,\infty)\}$ is a subset of the compact set $\{f(\ue^{2\pi \theta \ui}\argument)\colon \theta \in [0,1]\}$ and thus, the orbit is relatively compact. However, we can see by considering a sequence $(q_n)_{n \in \bbN}$ as in Example~\ref{ex:easy-spectrum-not-sufficient-for-semigroups-with-real-times} that the semigroup at infinity is not compact.
	\end{enumerate}
\end{remark}

Example~\ref{ex:easy-spectrum-not-sufficient-for-semigroups-with-real-times} shows that, if the semigroup at infinity associated to $(T_{s_0}^n)_{n \in \bbN_0}$ is non-empty and compact for each $s_0 \in [0,\infty)$, we cannot automatically conclude that the semigroup at infinity associated to $(T_s)_{s \in [0,\infty)}$ is non-empty and compact. If we want this implication to be true we need an additional assumption, and this is the only time in the theoretical part of this paper where we are forced to impose a time regularity condition on our semigroup. In fact, if the semigroup is strongly continuous at a strictly positive time, we obtain the following characterisation.

\begin{theorem} \label{thm:from-subsemigroups-to-the-semigroup-via-strong-continuity}
	Let $E$ be a Banach space and let $(T_s)_{s \in [0, \infty)}$ be a bounded semigroup on $E$ which is strongly continuous at at least one time $s_0 \in (0,\infty)$. The following assertions are equivalent:
	\begin{enumerate}
		\item[\upshape (i)] For each $s \in (0,\infty)$ the semigroup at infinity associated with $(T_s^n)_{n \in \bbN_0}$ is non-empty and compact.
		
		\item[\upshape (ii)] The semigroup at infinity associated with $(T_s)_{s \in [0,\infty)}$ is non-empty and compact.
	\end{enumerate}
	If the underlying scalar field of $E$ is complex, the above assertions~\textup{(i)} and \textup{(ii)} are also equivalent to:
	\begin{enumerate}
		\item[\upshape (iii)] For each $s \in (0,\infty)$ all spectral values of $T_s$ on the complex unit circle are poles of the resolvent of $T_s$.
	\end{enumerate}
\end{theorem}

The proof of Theorem~\ref{thm:from-subsemigroups-to-the-semigroup-via-strong-continuity} requires a bit of preparation. Let $\varphi\colon K \to K$ be a continuous map on some compact Hausdorff space $K$. In this case, the pair $(K; \varphi)$ is called a \emph{topological dynamical system}. Further, a point $x \in K$ is called \emph{recurrent} for the system $(K; \varphi)$ if for each neighbourhood $U \subseteq K$ of $x$ there is $n \in \bbN$ such that $\varphi^n(x) \in U$. It is not hard to see that $x \in K$ is recurrent if and only if $x \in K$ is \emph{infinitely recurrent}, that is for each neighbourhood $U \subseteq K$ of $x$ and each $n_0 \in \bbN$ there is $n \in \bbN$ with $n \ge n_0$ such that $\varphi^n(x) \in U$. More facts on recurrence in topological dynamical systems can for instance be found in \cite[Chapter~3.2]{Eisner2015}. We now use these notions to prove the following lemma.

\begin{lemma} \label{lem:unit-circle-pointwise}
	There exists a cofinal net $(n_j)_j$ in $\bbN$ such that the net $(\lambda^{n_j})_j$ converges to $1$ for each $\lambda \in \bbT$. 
\end{lemma}
\begin{proof}
	Endow $G \coloneqq \bbT^\bbT$ with the topology of pointwise convergence and with the pointwise multiplication. Then $G$ is a compact topological group. Set $\one \coloneqq (1)_{\lambda \in \bbT}$ and let $\varphi\colon G \to G$ be given by $\varphi(\mu) = (\lambda \mu_\lambda)_{\lambda \in \bbT}$ for each $\mu = (\mu_\lambda)_{\lambda \in \bbT}$. Then $\varphi$ is continuous and the topological dynamical system $(G; \varphi)$ is a so-called \emph{group rotation}. Hence, by \cite[Proposition~3.12(d)]{Eisner2015} every point in $G$ is recurrent with respect to $(G,\varphi)$ and thus, so is the point $\one$.
	
	Now, let $\calU$ denote the neighbourhood filter of $\one$ in $G$, ordered by converse set inclusion, and endow $\calU \times \bbN$ with the product order, which renders it a directed set. For each pair $(U,k) \in \calU \times \bbN$ we can find a number $n_{(U,k)} \in \bbN$ such that $n_{(U,k)} \ge k$ and $\varphi^{n_{(U,k)}}(\one) \in U$. Hence, the net $\big( \varphi^{n_{(U,k)}}(\one)\big)_{(U,k) \in \calU \times \bbN}$ converges to $\one$ in $G$, which means that $\big( \lambda^{n_{(U,k)}} \big)_{(U,k) \in \calU \times \bbN}$ converges to $1$ for each $\lambda \in \bbT$. Moreover, the net $\big(n_{(U,k)}\big)_{(U,k) \in \calU \times \bbN}$ is clearly cofinal in $\bbN$ by construction.
\end{proof}

Now we can show that, if the semigroup at infinity of a time-discrete operator semigroup $(T^n)_{n \in \bbN_0}$ is non-empty and compact, then there exists a subnet $(T^{n_j})_j$ which converges to $P_\infty$, where $(n_j)_j$ can be chosen independently of the operator $T$ (and also independently of the underlying Banach space).

\begin{proposition} \label{prop:single-net-for-all-operators}
	Let $(n_j)_j$ be a cofinal net in $\bbN$ such that $(\lambda^{n_j})_j$ converges to $1$ for each $\lambda \in \bbT$ (such a net exists according to Lemma~\ref{lem:unit-circle-pointwise}). If $(T^n)_{n \in \bbN_0}$ is a bounded semigroup on a Banach space $E$ whose semigroup at infinity is non-empty and compact, then $(T^{n_j})_j$ converges to $P_\infty$.
\end{proposition}
\begin{proof}
	We may assume throughout the proof that the scalar field is complex, since otherwise we may replace $E$ with a complexification. We know from Proposition~\ref{prop:compactness-for-single-operators} that $P_\infty$ is the spectral projection of $T$ associated with $\spec(T) \cap \bbT$. Since the net $(n_j)_j$ is cofinal in $\bbN$, Theorem~\ref{thm:JdLG-semigroup-infinity}(c) yields $(T|_{\ker P_\infty})^{n_j} \to 0$. Moreover, $E_\infty$ can be decomposed as
	\begin{align*}
		E_\infty = P_1 E \oplus \dots \oplus P_m E,
	\end{align*}
	where $m \in \bbN_0$, $\spec(T) \cap \bbT = \{\lambda_1,\dots,\lambda_m\}$ and $P_1,\dots,P_m$ are the spectral projections associated with the single spectral values $\lambda_1,\dots, \lambda_m$. The operator $T$ acts on the space $E_\infty$ as the multiplication with the tuple $(\lambda_1,\dots,\lambda_m)$, so it follows readily that $(T|_{E_\infty})^{n_j} \to \id_{E_\infty}$.
\end{proof}

\begin{proof}[Proof of Theorem~\ref{thm:from-subsemigroups-to-the-semigroup-via-strong-continuity}]
	We may assume throughout the proof that $E$ is a complex Banach space since we can otherwise replace $E$ with a complexification. 
	
	(i) $\Leftrightarrow$ (iii): This equivalence follows from Proposition~\ref{prop:compactness-for-single-operators}.
	
	(ii) $\Rightarrow$ (i): This implication follows from Corollary~\ref{cor:subsemigroup-non-empty-and-compact}. 
	
	(i) $\Rightarrow$ (ii): For each $s \in (0, \infty)$ denote by $P_{\infty,s}$ the projection at infinity that belongs to the semigroup $(T_{ns})_{n \in \bbN_0}$; then $P_{\infty,s}$ is also the spectral projection of $T_s$ that belongs to the intersection of the spectrum with the unit circle. Let $(n_j) \subseteq \bbN$ be a cofinal net with the property asserted in Lemma~\ref{lem:unit-circle-pointwise}. According to Proposition~\ref{prop:single-net-for-all-operators} we have $T_{n_js} = T_s^{n_j} \to P_{\infty, s}$ for each $s \in (0,\infty)$, which implies that the operator family $(P_{\infty, s})_{s \in (0,\infty)}$ satisfies the semigroup law. This in turn implies that all the projections $P_{\infty,s}$ coincide (see \cite[Lemma~2.2]{GerlachLB}); from now on, we set $P \coloneqq P_{s,\infty}$ for all $s \in (0,\infty)$. Since all operators $T_s$ commute with $P$, our the semigroup $(T_s)_{s \in [0,\infty)}$ leaves both $\ker P$ and $PE$ invariant. It remains to prove that $P$ satisfies the conditions~(a) and~(b) of Proposition~\ref{prop:uniqueness-of-proposition}:

	(a) It is an immediate consequence of Proposition~\ref{prop:bounded-convergence} that $(T_s|_{\ker P})_{s \in [0,\infty)}$ converges to $0$ as, for instance, the powers of $T_1|_{\ker P} = T_1|_{\ker P_{\infty,1}}$ converge to $0$.

	(b) It follows from Theorem~\ref{thm:JdLG-semigroup-infinity}(b) that, for each $s \in (0,\infty)$, the operator $T_s|_{PE}$ is invertible on $PE$. Hence, the semigroup $(T_s|_{PE})_{s \in [0, \infty)}$ extends to a group on $PE$. Since the semigroup is strongly continuous at at least one time, it thus follows that it is strongly continuous at all times $s \in [0,\infty)$. Let $A$ denote the generator of the $C_0$-semigroup $(T_s|_{PE})_{s \in [0, \infty)}$.
	
	Let us show that the operator $A$ has at most finitely many eigenvalues on the imaginary axis. So assume to the contrary that the set $\ui B \coloneqq \spec_\pnt(A) \cap \ui \bbR$ is infinite. Choose two time $s,t \in (0,\infty)$ such that $s/t$ is irrational. Since $\ue^{\ui tB}$ consists of unimodular eigenvalues of $T_t|_{PE}$, it follows that this set is finite. Hence, there exists an infinite subset $\ui C$ of $\ui B$ whose values are all mapped to the same number by the mapping $\exp(\argument t)$. Thus, $t(c_1-c_2) \in 2\pi \bbZ$ for all $c_1,c_2 \in C$. Consequently, $s(c_1-c_2) = \frac{s}{t}t(c_1-c_2) \not\in 2\pi \bbZ$ for any two distinct $c_1,c_2 \in C$, which conversely implies that all the values $\ue^{\ui sc}$ are distinct for $c \in C$. However, each such number is an eigenvalue of $T_s|_{PE}$; this is a contradiction since $T_s|_{PE}$ has only finitely many eigenvalues.
	
	Let $\ui \beta_1, \dots, \ui \beta_n$ denote the eigenvalues of $A$ on the imaginary axis (at least one such eigenvalue exists unless $PE = \{0\}$) and denote their corresponding eigenspaces by $E_1,\dots,E_n$. We note that $PE = E_1 \oplus \dots \oplus E_n$. To see this, choose a sufficiently small number $s_0 \in (0,\infty)$ such that all the numbers $\ue^{\ui s_0\beta_1},\dots, \ue^{\ui s_0\beta_n}$ are distinct. Then, for each $k\in \{1,\dots,n\}$, the space $E_k$ is the eigenspace of $T_{s_0}$ for the eigenvalue $\ue^{\ui s_0\beta_k}$ \cite[Corollary~IV.3.8(ii)]{Engel2000}. Consequently, $E_k$ is even the spectral space of $T_{s_0}$ for the spectral value $\ue^{\ui s_0\beta_k}$ since the latter number is a first order pole of the resolvent of $T_{s_0}$ (as $T_{s_0}$ is power-bounded). Moreover, $P$ is the spectral projection of $T_{s_0}$ corresponding to the part $\spec(T_{s_0}) \cap \bbT = \{\ue^{\ui s_0\beta_1},\dots,\ue^{\ui s_0\beta_n}\}$ of the spectrum, so indeed
	\begin{align*}
		PE = \ker(\ue^{\ui s_0\beta_k} - T_{s_0}) \oplus \dots \oplus \ker(\ue^{\ui s_0\beta_k} - T_{s_0}) = E_1 \oplus \dots \oplus E_n.
	\end{align*}
	As the semigroup $(T_s|_{PE})_{s \in [0,\infty)}$ acts on $E_k$ as the multiplication with $(\ue^{\ui s \beta_k})_{s \in [0,\infty)}$, it follows that $\{T_s|_{PE} \colon s \in [0, \infty) \}$ is relatively compact in $\calL(PE)$. 
\end{proof}

\section{Triviality of compact operator groups} \label{section:triviality-of-compact-operator-groups}

Loosely speaking, the major theoretical consequence of Corollary~\ref{cor:characterization-of-sg-convergence} is that, if one would like to find sufficient criteria for an operator semigroup to converge with respect to the operator norm, then one should seek for criteria which ensure that a compact operator group is trivial. This is the purpose of the present section.

\subsection{Divisible groups and a spectral condition} \label{subsection:divisible-groups-and-a-spectral-condition}

Recall that a compact topological group $G$ is called \textit{divisible} if for each $g \in G$ and each $n \in \bbN$, there exists $h \in G$ such that $h^n = g$. We start with a theorem on the triviality of divisible compact groups of linear operators. The corollaries of this theorem that are listed at the end of this subsection will be powerful tools in Section~\ref{section:operator-norm-convergence-of-semigroups} when we finally prove various concrete convergence theorems for operator semigroups.

It is worthwhile to note that a compact topological group $G$ is divisible if and only if $G$ is connected (see \cite[Corollary~2]{Mycielski1958} or, for the special case where $G$ is commutative, \cite[assertions~(a) and~(b) on p.\,55]{Kaplansky1954}).

\begin{theorem} \label{thm:trivial-group-by-rational-spectrum}
	Let $E$ be a complex Banach space and let $\calG \subseteq \calL(E)$ be a divisible and compact subgroup of the invertible linear operators on $E$. If, for each $T \in \calG$, all spectral values of $T$ are roots of unity, then $\calG = \{\id_E\}$.
\end{theorem}

For the proof of Theorem~\ref{thm:trivial-group-by-rational-spectrum} we need a bit of Banach algebra theory, specifically the following lemma. For the convenience of the reader, we include its simple proof.

\begin{lemma} \label{lemma:maximal-commutative-subalgebra}
	Let $\calA$ be a complex Banach algebra with multiplicatively neutral element $1$ and let $\calB \subseteq \calA$ be a commutative subalgebra which is maximal among all commutative subalgebras of $\calA$. Then the following assertions hold:
	\begin{enumerate}[\upshape (i)]
		\item $\calB$ is closed and contains $1$.
		\item For each $b \in \calB$ its spectrum in $\calA$ coincides with its spectrum in $\calB$. 
	\end{enumerate}
\end{lemma}
\begin{proof}
	(i) This follows immediately from the maximality of $\calB$.

	(ii) Fix $b \in B$. Clearly, the spectrum of $b$ in $\calA$ is contained in the spectrum of $b$ in $\calB$. To show the converse inclusion, let $\lambda$ be in the resolvent set of $b$ with respect to $\calA$. Observe that the inverse $(\lambda - b)^{-1}$ commutes with all elements in $\calB$. Therefore, the linear span of the set
\begin{align*}
	\{(\lambda - b)^{-n} \colon n \in \bbN_0 \} \cdot \calB
\end{align*}
is a commutative subalgebra of $\calA$ that contains $\calB$ and thus coincides with $\calB$. Hence, $(\lambda - b)^{-1} \in \calB$, i.e., $\lambda$ is contained in the resolvent set of $b$ in $\calB$.
\end{proof}

Note that if $\calG \subseteq \calL(E)$ is a compact subgroup of the invertible linear operators on a complex Banach space $E$, then $\sup_{n \in \bbZ} \norm{T^n} < \infty$ for all $T \in \calG$, i.e., each operator in $\calG$ is \emph{doubly power-bounded}. After these preparations, Theorem~\ref{thm:trivial-group-by-rational-spectrum} can be proved.

\begin{proof}[Proof of Theorem~\ref{thm:trivial-group-by-rational-spectrum}]
	According to \cite[Corollary~1 and Corollary~2]{Mycielski1958} every element of a divisible compact group is contained in a divisible commutative (and closed) subgroup, so it suffices to prove the assertion for commutative $\calG$.
	
	Let $\calB$ be a subalgebra of $\calL(E)$ which is maximal among all commutative subalgebras of $\calL(E)$ that contain $\calG$ (such a $\calB$ exists by Zorn's lemma). Then $\calB$ is also maximal among all commutative subalgebras of $\calL(E)$, so according to Lemma~\ref{lemma:maximal-commutative-subalgebra}, $\calB$ is closed and contains $\id_E$; moreover, for each $T \in \calB$ the spectrum $\spec(T)$ of $T$ in $\calL(E)$ and its spectrum in $\calB$ coincide. Hence, if $\Omega(\calB)$ denotes the character space of the Banach algebra $\calB$, then we have
	\begin{align*}
		\spec(T) = \{\varphi(T)\colon \varphi \in \Omega(\calB)\};
	\end{align*}
	see e.g.~\cite[Theorem~1.3.4(1)]{Murphy1990}. Since each $\varphi \in \Omega(\calB)$ is a continuous group homomorphism, it follows that $\varphi(\calG)$ is a divisible and compact subgroup of $\bbT$ for each character $\varphi$. On the other hand, it follows from our spectral assumption that $\varphi(\calG)$ consists of roots of unity only; consequently, $\varphi(\calG) = \{1\}$. We conclude that $\spec(T) = \{1\}$ for each $T \in \calG$, so each such $T$ equals $\id_E$ by Gelfand's $T = \id$ theorem since $T$ is doubly power-bounded (see e.g.\ \cite[Theorem~B.17]{Engel2000}).
\end{proof}

The Banach algebra argument used in the previous proof is a common technique in the spectral analysis of operator semigroups; related arguments can, for instance, be found in \cite[Chapter~XVI]{Hille1957} and \cite[Section~4.7]{Blake1999}.

The condition that all spectral values of any $T \in \calG$ are roots of unity is automatically satisfied in two important situations. The first one is that the underlying space is a Banach lattice and all operators in $\calG$ are positive; this is the content of the following corollary.

\begin{corollary} \label{cor:trivial-group-positive}
	Let $E$ be a Banach lattice and let $\calG \subseteq \calL(E)$ be a divisible and compact subgroup of the invertible bounded linear operators on $E$ such that each operator in $\calG$ is positive. Then $\calG = \{\id_E\}$.
\end{corollary}
\begin{proof}
	One may assume that the scalar field is complex. According to Theorem~\ref{thm:trivial-group-by-rational-spectrum} it suffices to show that the spectrum of each $T \in \calG$ consists of roots of unity only, so fix $T \in \calG$. Clearly, $\spec(T) \subseteq \bbT$, so it follows from Lemma~\ref{lemma:compactness-discrete-case} and Proposition~\ref{prop:compactness-for-single-operators} that $\spec(T)$ is finite and consists of poles of the resolvent.
	
	It follows from infinite-dimensional Perron--Frobenius theory that the spectrum of $T$ is \emph{cyclic}, meaning that $\lambda^n \in \spec(T)$ for all $n \in \bbZ$ whenever $\lambda \in \spec(T)$ (see \cite[Theorem~4.7]{Lotz1968} or \cite[Theorem~V.4.9]{Schaefer1974}). By the finiteness of the spectrum, this implies that $\spec(T)$ consists of roots of unity only.
\end{proof}

Before we continue with a second situation where Theorem~\ref{thm:trivial-group-by-rational-spectrum} can be applied, let us briefly discuss another possibility to derive the corollary above. The following theorem also contains Corollary~\ref{cor:trivial-group-positive} as a special case since every finite divisible group consists of one element only. The theorem and its proof were communicated to us by Rainer Nagel, according to whom the result goes originally back to Heinrich P.\ Lotz. We could not find a concrete reference for it in the literature, though.

\begin{theorem}
	Let $E$ be a (real or complex) Banach lattice and let $\calG \subseteq \calL(E)$ be a compact subgroup of the invertible bounded linear operators on $E$ such that each operator in $\calG$ is positive. Then $\calG$ is finite.
\end{theorem}
\begin{proof}
	We may assume that the underlying scalar field is complex. Since $\calG$ is compact, it suffices to show that every element in $\calG$ is isolated, and to this end it suffices to prove that $I$ is isolated. 
	
	Now, fix $T \in \calG \setminus \{I\}$. Since $T$ is doubly power bounded, the spectrum of $T$ is a subset of the complex unit circle. Moreover, the spectrum cannot consist of the number $1$ only, since this would imply $T = \id$ by Gelfand's $T = \id$ theorem \cite[Theorem~B.17]{Engel2000}. Since the peripheral spectrum (which is the spectrum) of $T$ is cyclic (see \cite[Theorem~4.7]{Lotz1968} or \cite[Theorem~V.4.9]{Schaefer1974}), there exists a spectral value $\lambda$ of $T$ with negative real part. In particular, the spectral value $\lambda - 1$ of the operator $T-\id$ has modulus at least $\sqrt{2}$, so
	\begin{align*}
		\norm{T - \id} \ge r(T-\id) \ge \modulus{\lambda - 1} \ge \sqrt{2}.
	\end{align*}
	This shows that every operator in $\calG \setminus \{\id\}$ has distance at least $\sqrt{2}$ from $\id$, so $\id$ is indeed isolated in $\calG$.
\end{proof}

After this brief intermezzo, let us continue to discuss consequences of Theorem~\ref{thm:trivial-group-by-rational-spectrum}. Our next corollary deals with the case of contractive operators on so-called \emph{projectively non-Hilbert spaces}. This notion is taken from \cite[Definition~3.1]{Glueck2016b}; a real Banach space $E$ is called \emph{projectively non-Hilbert}, if for no rank-$2$ projection $P \in \calL(E)$, the range $PE$ is isometrically a Hilbert space. Every real-valued $L^p$-space over an arbitrary measure space is projectively non-Hilbert if $p \in [1,\infty] \setminus \{2\}$, see \cite[Example~3.2]{Glueck2016b} and the discussion after \cite[Example~3.5]{Glueck2016b}. Moreover, every real Banach lattice that is a so-called \emph{AM}-space is projectively non-Hilbert \cite[Example~1.2.7]{GlueckDISS}; this includes the space of real-valued bounded and continuous functions on any topological space.

\begin{corollary} \label{cor:trivial-group-contractive}
	Let $E$ be a real Banach space that is projectively non-Hilbert and let $\calG \subseteq \calL(E)$ be a divisible and compact subgroup of the invertible bounded linear operators on $E$ such that each operator in $\calG$ is contractive. Then $\calG = \{\id_E\}$.
\end{corollary}
\begin{proof}
	Let $E_\bbC$ denote a Banach space complexification of $E$; for each $T \in \calG$ we denote the canonical extension of $T$ to $E_\bbC$ by $T_\bbC$. Then $\calG_\bbC \coloneqq \{T_\bbC\colon T \in \calG\}$ is a divisible and compact subgroup of the invertible bounded linear operators on $E_\bbC$. 
	
	Now fix $T \in \calG$; it suffices to prove that the spectrum of $T_\bbC$ consists of roots of unity only. By Proposition~\ref{prop:characterisation-of-compact-and-non-empty} the semigroup at infinity associated to $(T_\bbC^n)_{n \in \bbN_0}$ is non-empty and compact, so it follows from Proposition~\ref{prop:compactness-for-single-operators} that $\spec(T_\bbC)$ is a finite subset of the complex unit circle and consists of eigenvalues of $T_\bbC$. Moreover, the set $\{T_\bbC^n\colon n \in \bbN_0\}$ is relatively compact with respect to the weak operator topology, i.e., $T_\bbC$ is \emph{weakly almost periodic}. Since $E$ is projectively non-Hilbert, we can now apply \cite[Theorem~3.11]{Glueck2016b} to conclude that the spectrum of $T_\bbC$ consists of roots of unity only.
\end{proof}

\subsection{Strong positivity of groups} \label{subsection:strong-positivity-of-groups}

Another way to ensure that a group of linear operators is trivial is to ensure a certain condition of \emph{strong positivity}; this works in the very general setting of ordered Banach spaces. By an \emph{ordered Banach space} we mean a tuple $(E,E_+)$ where $E$ is a real Banach space and $E_+$ is a closed subset of $E$ such that $\alpha E_+ + \beta E_+ \subseteq E_+$ for all $\alpha,\beta \in [0,\infty)$ and such that $E_+ \cap (-E_+) = \{0\}$; the set $E_+$ is called the \emph{positive cone} in $E_+$.

Let $(E,E_+)$ be an ordered Banach space. An operator $T \in \calL(E)$ is called \emph{positive} if $TE_+ \subseteq E_+$; a semigroup on $E$ is said to be \emph{positive} if every operator in it is positive. A functional $\varphi \in E'$ is called \emph{positive} if $\langle \varphi, f \rangle \ge 0$ for all $f \in E_+$. A vector $f \in E_+$ is said to be an \emph{almost interior point} of $E_+$ if $\langle \varphi, f \rangle > 0$ for each non-zero positive functional $\varphi \in E'$. If, for instance, $E$ is an $L^p$-space over a $\sigma$-finite measure space and $p \in [1,\infty)$, then a function $f \in E_+$ is an almost interior point if and only if $f(\omega) > 0$ for almost all $\omega \in \Omega$. For more information about almost interior points we refer to \cite[Section~2]{GlueckAlmostInterior}. The following result is inspired by the proof of \cite[Theorem~4.1]{GlueckAlmostInterior}.

\begin{theorem} \label{thm:trivial-group-strictly-positive}
	Let $(E,E_+)$ be an ordered Banach space with $E_+ \not= \{0\}$ and let $\calG \subseteq \calL(E)$ be a norm-bounded subgroup of the invertible operators on $E$. Assume that every operator in $\calG$ is positive and that, for each $f \in E_+ \setminus \{0\}$, there exists $T \in \calG$ such that $Tf$ is an almost interior point of $E_+$. Then $E$ is one-dimensional and $\calG = \{\id_E\}$.
\end{theorem}
\begin{proof}
	We first show that every point in $E_+ \setminus \{0\}$ is an almost interior point of $E_+$. So let $f \in E_+ \setminus \{0\}$. Choose $T \in \calG$ such that $Tf$ is an almost interior point of $E_+$. Since $T^{-1}$ is an element of $\calG$, it is a positive operator on $E$, and since $T^{-1}$ is surjective it thus follows from \cite[Corollary~2.22(a)]{GlueckAlmostInterior} that $T^{-1}$ maps almost interior points to almost interior points. Hence, $f = T^{-1}Tf$ is an almost interior point.
	
	Since all vectors in $E_+ \setminus \{0\}$ are almost interior points, it follows from \cite[Theorem~2.10]{GlueckAlmostInterior} that $E$ is one-dimensional. Thus, $\calG$ can be identified with a bounded subgroup of the multiplicative group $(0,\infty)$, so $\calG$ does indeed consist of one element only.
\end{proof}

\section{Operator norm convergence of semigroups} \label{section:operator-norm-convergence-of-semigroups}

In this section we finally derive convergence theorems for various classes of operator semigroups. In Subsection~\ref{subsection:convergence-under-divisibility-conditions} representations whose underlying semigroup $(S,+)$ satisfies a certain kind of divisibility condition are considered. In Subsection~\ref{subsection:convergence-under-a-strict-positivity-condition} we then deal with positive semigroups on ordered Banach spaces under an appropriate strong positivity assumption.

\subsection{Convergence under divisibility conditions} \label{subsection:convergence-under-divisibility-conditions}

We call the semigroup $(S,+)$ \emph{essentially divisible} if, for each $s \in S$ and each integer $n \in \bbN$, there exist elements $t_1,t_2 \in S$ such that $nt_1 = s + nt_2$. This definition is taken from \cite{Glueck2019}, where it was used as a generalisation of semigroups that generate divisible groups (which played an important role in \cite{GerlachConvPOS}). Let us illustrate the notion of essential divisibility with a list of simple examples.

\begin{examples}
	\begin{enumerate}[\upshape (a)]
		\item The semigroup $([0, \infty),+)$ is essentially divisible, and so is $(\bbQ \cap [0,\infty), +)$.
		\item More generally, for each $a \ge 0$, both the semigroup $(\{0\} \cup [a,\infty),+)$ and the semigroup $\big(\{0\} \cup (\bbQ \cap [a,\infty)), +\big)$ are essentially divisible.
		\item The semigroup $([0,\infty)^n,+)$ is essentially divisible for each $n \in \bbN$.
		\item The semigroup $([0,\infty), \max)$ is essentially divisible; here, $\max$ denotes the binary operator which assigns the maximum to any two given elements of $[0,\infty)$.
		\item More generally, if $L$ is a lattice with a smallest element $i$, then $(L,\lor)$ is an essentially divisible semigroup (with neutral element $i$).
		\item The semigroup $(\bbN_0,+)$ is not essentially divisible.
		\item The semigroup $(D,+)$, where $D = \{k/2^n\colon k,n \in \bbN_0\}$ is the set of dyadic numbers in $[0,\infty)$, is not essentially divisible.
	\end{enumerate}
\end{examples}

Now we use the notion of essential divisibility to prove a convergence theorem for positive semigroups on Banach lattices and a convergence theorem for contractive semigroups on projectively non-Hilbert spaces. Let us begin with the positive case.

\begin{theorem} \label{thm:bounded-convergence-banach-lattice}
	Let $E$ be a Banach lattice and let $(T_s)_{s \in S}$ be a positive and bounded semigroup on $E$. If the semigroup at infinity, $\sgInftyON{\calT}$, is non-empty and compact and if $(S,+)$ is essentially divisible, then $(T_s)_{s \in S}$ converges with respect to the operator norm to the projection at infinity.
\end{theorem}
\begin{proof}
	Note that the range $E_\infty$ of the projection at infinity, $P_\infty$, is again a Banach lattice since $P_\infty$ is positive \cite[Proposition~II.11.5]{Schaefer1974}. According to Theorem~\ref{thm:JdLG-semigroup-infinity},
	\begin{align*}
		\sgInftyON{\calT}|_{E_\infty} = \overline{\{T_s \colon s \in S\}|_{E_\infty}}^{\calL(E_\infty)}
	\end{align*}
	is a compact subgroup of the invertible operators on $E_\infty$. As $(S, +)$ is essentially divisible, a simple compactness argument thus shows that $\sgInftyON{\calT}|_{E_\infty}$ is divisible. Since this group consists of positive operators, it is therefore trivial by Corollary~\ref{cor:trivial-group-positive}.
	
	Since the groups $\sgInftyON{\calT}|_{E_\infty}$ and $\sgInftyON{\calT}$ are isomorphic by Theorem~\ref{thm:JdLG-semigroup-infinity}, the semigroup at infinity, $\sgInftyON{\calT}|_{E_\infty}$, is also trivial. Thus, Corollary~\ref{cor:characterization-of-sg-convergence} yields the claim.
\end{proof}

The following corollary is due to Lotz in the special case where $S = [0,\infty)$.

\begin{corollary} \label{cor:quasi-compact-convergence-banach-lattice}
	Let $E$ be a Banach lattice and let $(T_s)_{s \in S}$ be a positive and bounded semigroup on $E$. If $T_{s_0}$ is quasi-compact for at least one $s_0 \in S$ and if $(S,+)$ is essentially divisible, then $(T_s)_{s \in S}$ converges with respect to the operator norm to a finite rank projection.
\end{corollary}
\begin{proof}
	This is an immediate consequence of Corollary~\ref{prop:quasi-compact} and Theorem~\ref{thm:bounded-convergence-banach-lattice}.
\end{proof}

Our second corollary -- which only deals with the semigroup $([0,\infty),+)$ -- has the nice theoretical feature that it covers, in contrast to Corollary~\ref{cor:quasi-compact-convergence-banach-lattice}, also the trivial operator semigroup that consists merely of the operator $\id_E$ -- which is arguably the most simple convergent operator semigroup.

\begin{corollary} \label{cor:spectral-condition-banach-lattice}
	Let $E$ be a complex Banach lattice and let $(T_s)_{s \in [0,\infty)}$ be a positive and bounded semigroup on $E$ which is strongly continuous at at least one time $s_0 \in (0,\infty)$. If, for each $s \in (0,\infty)$, all spectral values of $T_s$ on the unit circle are poles of the resolvent, then $T_s$ converges with respect to the operator norm as $s \to \infty$.
\end{corollary}
\begin{proof}
	This is an immediate consequence of Theorems~\ref{thm:from-subsemigroups-to-the-semigroup-via-strong-continuity} and~\ref{thm:bounded-convergence-banach-lattice}.
\end{proof}

Now we deal with real Banach spaces which are projectively non-Hilbert; see the discussion before Corollary~\ref{cor:trivial-group-contractive} for a definition of this property.

\begin{theorem} \label{thm:convergence-projectively-non-Hilbert-case}
	Let $E$ be a real Banach space that is projectively non-Hilbert and let $(T_s)_{s \in S}$ be a contractive semigroup on $E$. If the semigroup at infinity, $\sgInftyON{\calT}$, is non-empty and compact and if $(S,+)$ is essentially divisible, then $(T_s)_{s \in S}$ converges with respect to the operator norm to the projection at infinity.
\end{theorem}
\begin{proof}
	Note that the semigroup at infinity, $P_\infty$, is contractive, and hence its range is itself a projectively non-Hilbert space. It follows as in Theorem~\ref{thm:bounded-convergence-banach-lattice} that the compact group $\sgInftyON{\calT}|_{E_\infty}$ is divisible; since it consists of contractive operators only, Corollary~\ref{cor:trivial-group-contractive} shows that this group is actually trivial. Thus, the semigroup at infinity -- which is isomorphic to $\sgInftyON{\calT}|_{E_\infty}$ -- is trivial, too. So the conclusion follows from Corollary~\ref{cor:characterization-of-sg-convergence}.
\end{proof}

Again, we state the same result separately for the quasi-compact case.

\begin{corollary} \label{cor:quasi-compact-convergence-projectively-non-Hilbert}
	Let $E$ be a real Banach space that is projectively non-Hilbert and let $(T_s)_{s \in S}$ be a contractive semigroup on $E$. If $T_{s_0}$ is quasi-compact for at least one $s_0 \in S$ and if $(S,+)$ is essentially divisible, then $(T_s)_{s \in S}$ converges with respect to the operator norm to a finite rank projection.
\end{corollary}
\begin{proof}
	This is an immediate consequence of Corollary~\ref{prop:quasi-compact} and Theorem~\ref{thm:convergence-projectively-non-Hilbert-case}.
\end{proof}

A similar result as in Corollary~\ref{cor:spectral-condition-banach-lattice} is, of course, also true for contractive semigroups on projectively non-Hilbert spaces; we refrain from stating this explicitly as a corollary.

Finally, Theorem~\ref{thm:introduction} from the introduction follows readily from Corollary~\ref{cor:quasi-compact-convergence-projectively-non-Hilbert}:

\begin{proof}[Proof of Theorem~\ref{thm:introduction}]
	(i) $\Rightarrow$ (ii): This implication is obvious.
	
	(ii) $\Rightarrow$ (i): For both possible choices of $E$, this space is projectively non-Hilbert. Since the semigroup $([0,\infty),+)$ is essentially divisible, the assertion follows from Corollary~\ref{cor:quasi-compact-convergence-projectively-non-Hilbert}.
\end{proof}

\begin{remark}
	\label{rem:continuous-vs-discrete-time}
	All results in this subsection fail as we drop the assumption that the semigroup $(S,+)$ is essentially divisible. For instance, the semigroup $(\bbN_0,+)$ is not essentially divisible, and indeed the $n$-th powers of the matrix
	\begin{align*}
		\begin{pmatrix}
			0 & 1 \\
			1 & 0
		\end{pmatrix}
	\end{align*}
	do not converge as $n \to \infty$ -- despite the fact that the matrix is positive and contractive with respect to the $p$-norm for each $p$. A closely related phenomenon is discussed in \cite[Example~3.7]{GerlachConvPOS}.
\end{remark}

\subsection{Convergence under a strong positivity condition} \label{subsection:convergence-under-a-strict-positivity-condition} 

The following theorem is generalisation of \cite[Theorem~5.3]{GlueckAlmostInterior} where only the cases $S = \bbN_0$ and $S = [0,\infty)$ where considered. 

\begin{theorem} \label{thm:convergence-of-a-strictly-positive-semigroup}
	Let $(E,E_+)$ be an ordered Banach space with $E_+ \not= \{0\}$ and let $(T_s)_{s \in S}$ be a bounded and positive semigroup on $E$. Moreover, assume that $T_{s_0}$ is quasi-compact for at least one $s_0 \in S$ and that the following strong positivity condition holds: for each $f \in E_+ \setminus \{0\}$ there exists $s \in S$ such that $T_sf$ is an almost interior point of $E_+$.
	
	Then $(T_s)_{s \in S}$ converges with respect to the operator norm to a projection in $\calL(E)$ of rank at most $1$.
\end{theorem}
\begin{proof}
	According to Proposition~\ref{prop:quasi-compact} the semigroup at infinity, $\sgInftyON{\calT}$, is non-empty and compact since $(T_s)_{s \in S}$ is bounded and since $T_{s_0}$ is quasi-compact. Let $P_\infty$ denote the corresponding projection at infinity. Then $P_\infty$ is a positive operator and hence, its range $E_\infty$ is also an ordered Banach space with positive cone $P_\infty E_+ = E_+ \cap E_\infty$. If $P_\infty = 0$, Theorem~\ref{thm:JdLG-semigroup-infinity}(c) implies that the semigroup converges to $0$; so assume now that $P_\infty \not= 0$.
	
	It follows from the assumptions that there exists at least one almost interior point in $E_+$, which implies that the set $E_+ - E_+$ is dense in $E$ (see e.g.\ \cite[Proposition~2.9]{GlueckAlmostInterior}). In particular, the positive cone $P_\infty E_+$ of the space $E_\infty$ is non-zero since $P_\infty \not= 0$.
	
	By Theorem~\ref{thm:JdLG-semigroup-infinity}(b), $\sgInftyON{\calT}|_{E_\infty}$ is a compact subgroup of the invertible operators on $E_\infty$, and for each $s \in S$ the restriction $T_s|_{E_\infty}$ is contained in $\sgInftyON{\calT}|_{E_\infty}$. Moreover, $\sgInftyON{\calT}|_{E_\infty}$ clearly consists of positive operators. We now show that this group satisfies the assumptions of Theorem~\ref{thm:trivial-group-strictly-positive}.
	
	To this end, let $0 \not= f \in P_\infty E_+$. By assumption there exists an $s \in S$ such that $T_sf$ is an almost interior point of $E_+$. Since $T_sf \in P_\infty E_+$, it follows from \cite[Corollary~2.22(b)]{GlueckAlmostInterior} that the vector $T_sf$ is also an almost interior point of the positive cone $P_\infty E_+$ of $E_\infty$. Hence, the operator $T_s|_{P_\infty E} \in \sgInftyON{\calT}|_{E_\infty}$ maps $f$ to an almost interior point of the positive cone of $E_\infty$, so we can employ Theorem~\ref{thm:trivial-group-strictly-positive} to conclude that $E_\infty$ is one-dimensional and that $\sgInftyON{\calT}|_{E_\infty} = \{\id_{E_\infty}\}$. Corollary~\ref{cor:characterization-of-sg-convergence} thus shows that $(T_s)_{s \in S}$ converges to the rank-$1$ projection $P_\infty$.
\end{proof}

\section{Application: coupled parabolic equations on $\bbR^d$} \label{section:application-coupled-parabolic-equations}

In this section we use Theorem~\ref{thm:introduction} to analyse the asymptotic behaviour of coupled parabolic equations with possibly unbounded coefficients on the space $\bbR^d$. Of course, the unboundedness of the coefficients forces us to impose other conditions on the equation in order to obtain well-posedness. Throughout the section we mainly rely on the results of \cite{Delmonte2011}, and as in this paper, we work on the space of bounded continuous functions over $\bbR^d$.

\subsection{Setting}

Here is our precise setting. Fix an integer $N \geq 1$ (which will denote the number of coupled equations) as well as functions $A \colon \bbR^d \to \bbR^{d \times d}$, $b\colon \bbR^d \to \bbR^d$ and $V\colon \bbR^d \to \bbR^{N \times N}$ and assume that the following conditions are satisfied: 

\begin{enumerate}[\upshape (a)]
	\item For all $x \in \bbR^d$ the matrix $A(x)$ is symmetric and there exists a continuous function $\nu\colon \bbR^d \to (0,\infty)$ such that the ellipticity condition
	\begin{align*}
		\xi^T A(x) \xi \ge \nu(x) \norm{\xi}_2
	\end{align*}
	holds for all $x \in \bbR^d$ and all $\xi \in \bbR^d$.
	
	\item There exists $\alpha \in (0,1)$ such that the functions $A$, $b$ and $V$ are locally $\alpha$-H\"older continuous on $\bbR^d$.
	
	\item The function $V$ is bounded.
	
	\item There exists a twice continuously differentiable function $\varphi\colon \bbR^d \to (0,\infty)$ such that $\varphi(x) \to \infty$ as $\norm{x}_2 \to \infty$ and a number $\lambda_0 > 0$ such that the estimate 
	\begin{align*}
		\lambda_0 \varphi - \sum_{i,j=1}^d A_{ij} \partial_{ij} \varphi - \sum_{j=1}^d b_j \partial_j \varphi \ge 0
	\end{align*}
	holds on $\bbR^d$. 
\end{enumerate}

Those are essentially the assumptions from~\cite[Hypotheses~2.1]{Delmonte2011}, with two exceptions:
\begin{itemize}
	\item Instead of boundedness of $V$ a weaker condition is used there (see \cite[Hypotheses~2.1(iii) and~Remark~2.2]{Delmonte2011}). The reason why we assume boundedness of $V$ is explained after Corollary~\ref{cor:semigroups-for-degenerate-parabolic-pde-contractive}.
	
	\item At first glance, the inequality in~\cite[Hypotheses~2.2(iv)]{Delmonte2011} looks slightly distinct from the inequality that is assumed in assertion~(d). However, since $V$ is assumed to be bounded, both inequalities are actually equivalent in our setting (if one changes $\lambda_0$ appropriately).
\end{itemize}

We point out that both $A$ and $b$ are allowed to be unbounded and that $A(x)$ need not be bounded away from $0$ as $\norm{x}_2 \to \infty$. In the following, the (possibly degenerate) parabolic equation
\begin{align}
	\label{eq:degenerate-parabolic-pde}
	\dot u = (\calB + V)u
\end{align}
is considered on the space $C_b(\bbR^d; \bbR^N)$ of bounded continuous function on $\bbR^d$ with values in $\bbR^N$, where the operator $\calB$ is given by
\begin{align}
	\label{eq:degenerate-elliptic-operator-without-potential}
	\calB u \coloneqq 
	\begin{pmatrix}
		\big(\sum_{i,j=1}^d A_{ij} \partial_{ij} + \sum_{j=1}^d b_j \partial_j\big) u_1 \\
		\vdots \\
		\big(\sum_{i,j=1}^d A_{ij} \partial_{ij} + \sum_{j=1}^d b_j \partial_j\big) u_N
	\end{pmatrix}
\end{align}
for all $u$ in the domain
\begin{align*}
	D(\calB) \coloneqq \{u \in C_b(\bbR^d; \bbR^N) & \cap \bigcap_{1 \le p < \infty} W^{2,p}_{\operatorname{loc}}(\bbR^d; \bbR^N)\colon \\ & \text{the expression in~\eqref{eq:degenerate-elliptic-operator-without-potential} is in } C_b(\bbR^d; \bbR^N)\}.
\end{align*}

The above setting will allow us to employ the results from \cite{Delmonte2011} about well-posedness of the equation~\eqref{eq:degenerate-parabolic-pde}. In order to apply our Theorem~\ref{thm:introduction} to study the long-term behaviour of the solutions, though, we have to ensure that the space $C_b(\bbR^d; \bbR^N)$ is isometrically isomorphic to a real-valued $C_b$-space. To this end, we endow it with the norm 
\begin{align*}
	\norm{u}_\infty = \max\{\norm{u_k}_\infty\colon k \in \{1,\dots,N\}\}
\end{align*}
for all $u$ in this space. This already suggests that, in order to apply Theorem~\ref{thm:introduction}, we further need the matrix $V(x)$ to be \emph{$\infty$-dissipative} for each $x \in \Omega$ as to ensure that the solution semigroup of~\eqref{eq:degenerate-parabolic-pde} is contractive. 

In \cite{Delmonte2011}, the space $C_b(\bbR^d; \bbR^N)$ is equipped with the norm $\norm{u} = \sum_{k=1}^N \norm{u_k}_\infty$ which is equivalent to the norm introduced above but which does not render $C_b(\bbR^d; \bbR^N)$ an AM-space.

\begin{proposition} \label{prop:semigroups-for-degenerate-parabolic-pde}
	The operators $\calB$ and $\calB + V$ \textup{(}with $D(\calB + V) \coloneqq D(\calB)$\textup{)} on $C_b(\bbR^d;\bbR^N)$ are closed, and all sufficiently large real numbers belong to the resolvent sets of both $\calB$ and $\calB+V$. 
	
	Moreover, there exist operator semigroups $(S_t)_{t \in [0,\infty)}$ and $(T_t)_{t \in [0,\infty)}$ on $C_b(\bbR^d; \bbR^N)$ with the following properties:
	\begin{enumerate}[\upshape (a)]
		\item For each $f \in C_b(\bbR^d;\bbR^N)$, each $x \in \bbR^d$ and all sufficiently large real numbers $\lambda$ the functions
		\begin{align*}
			(0,\infty) \ni t \mapsto \ue^{-\lambda t}S_tf(x) \in \bbR^N
			\quad \text{and} \quad
			(0,\infty) \ni t \mapsto \ue^{-\lambda t}T_tf(x) \in \bbR^N
		\end{align*}
		are continuous and in $L^1((0,\infty);\bbR^N)$, and their integrals equal $\Res(\lambda,\calB)f(x)$ and $\Res(\lambda,\calB+V)f(x)$, respectively.
		
		\item The semigroup $(S_t)_{t \in [0,\infty)}$ is contractive.
	\end{enumerate}
\end{proposition}
\begin{proof}
	The assertions about $\calB$ and $\calB + V$, as well as the existence of both semigroups and property~(a) follow from \cite[Section~3]{Delmonte2011}; to see that we can really use the domain $D(\calB)$ as domain of the operator $\calB+V$ we need the assumption that $V$ is bounded.
	
	Since $\calB$ acts separately in every component, so does the semigroup $(S_t)_{t \in [0,\infty)}$; hence, contractivity of $(S_t)_{t \in [0,\infty)}$ follows from contractivity in the scalar case, which can for instance be found in~\cite[Proposition~2.3(i)]{Delmonte2011}.
\end{proof}

The semigroup $(T_t)_{t \in [0,\infty)}$ describes the solutions to our parabolic equation~\eqref{eq:degenerate-parabolic-pde}; see \cite[Section~3]{Delmonte2011}. We note that, in our setting where the matrix potential $V$ is bounded, one could -- alternatively to the approach from \cite{Delmonte2011} -- employ the theory of bi-continuous semigroups to study the perturbed operator $\calB + V$; see \cite[beginning of Section~5]{Metafune2002} and \cite[Theorem~3.5]{Farkas2004}. (There are also results about unbounded perturbations of bi-continuous semigroups such as in \cite[Corollary~4.2]{Albanese2004}, but we do not know whether such results can be applied under the assumptions of \cite[Hypotheses~2.1]{Delmonte2011}).

We point out that while the semigroup $(S_t)_{t \in [0,\infty)}$ is positive, the semigroup $(T_t)_{t \in [0,\infty)}$ is not positive, in general. Moreover, in general we cannot expect those semigroups to be strongly continuous (see for instance the discussion at the beginning of \cite[Subsection~3.1]{Delmonte2011}).

If we assume $\infty$-dissipativity of the matrices $V(x)$, then the semigroup $(T_t)_{t \in [0,\infty)}$ is also contractive:

\begin{corollary} \label{cor:semigroups-for-degenerate-parabolic-pde-contractive}
	Assume that, for each $x \in \bbR^d$, the matrix $V(x)$ is dissipative with respect to the $\infty$-norm on $\bbR^N$. Then the semigroup $(T_t)_{t \in [0,\infty)}$ is contractive, too.
\end{corollary}
\begin{proof}
	For each $f \in C_b(\bbR^d;\bbR^N)$, each $x \in \bbR^d$ and each $\lambda > 0$ the mapping $(0,\infty) \ni t \mapsto \ue^{-\lambda t}S_tf(x) \in \bbR^N$ is continuous and in $L^1((0,\infty);\bbR^N)$, and its integral equals $\Res(\lambda,\calB)f(x)$; this follows from Proposition~\ref{prop:semigroups-for-degenerate-parabolic-pde} and from the identity theorem for analytic functions. As $(S_t)_{t \in [0,\infty)}$ is contractive, so is the operator $\lambda \Res(\lambda,\calB)$ for each $\lambda > 0$, and thus it follows that $\calB$ is dissipative. 
	
	The matrix-valued multiplication operator $V$ is dissipative by assumption, and since it is a bounded operator, it is thus even strictly dissipative. Consequently, the operator $\calB + V$ is dissipative, too. It now follows from Post's inversion formula for the Laplace transform (for $\bbR^N$-valued functions) and, again, from Proposition~\ref{prop:semigroups-for-degenerate-parabolic-pde} that $(T_t)_{t \in [0,\infty)}$ is contractive.
\end{proof}

The proof of Corollary~\ref{cor:semigroups-for-degenerate-parabolic-pde-contractive} is the reason why we assumed $V$ to be bounded; we needed the boundedness on two occasions in the proof: (i) in order to derive strict dissipativity of $V$ from mere dissipativity, and (ii) in order for $\calB + V$ to have the same domain as $\calB$. The authors do not know whether Corollary~\ref{cor:semigroups-for-degenerate-parabolic-pde-contractive} remains true for unbounded $V$ which satisfies, besides dissipativity, only the assumptions of \cite[Hypotheses~2.1]{Delmonte2011}.

\subsection{A convergence result}

After the preparations of the preceding subsection, we now arrive at the following convergence result for the solutions to~\eqref{eq:degenerate-parabolic-pde}. Let us remark that, if the matrices $V(x)$ in the potential have non-negative off-diagonal entries, the long-term behaviour of the solutions equations of the type~\eqref{eq:degenerate-parabolic-pde} was studied in~\cite[Section~4]{Addona2019}; this is possible since the mentioned assumption on $V(x)$ allows for the use of Perron--Frobenius theory. 

Here, we make no such positivity assumption. Instead, we are going to assume that the matrices $V(x)$ are $\infty$-dissipative. If the operator semigroup $(T_t)_{t \in [0,\infty)}$ is immediately compact, this implies that the solutions to~\eqref{eq:degenerate-parabolic-pde} converge uniformly (for initial values in the unit ball) as time tends to infinity.

\begin{theorem} \label{thm:convergence-on-unbounded-domains}
	Assume that, for each $x \in \bbR^d$, the matrix $V(x)$ is dissipative with respect to the $\infty$-norm on $\bbR^N$. If the operators $T_t$ are compact for $t > 0$, then $T_t$ converges with respect to the operator norm to a finite-rank projection as $t \to \infty$.
\end{theorem}
\begin{proof}
	This is a consequence of Corollary~\ref{cor:semigroups-for-degenerate-parabolic-pde-contractive} and Theorem~\ref{thm:introduction} since $C_b(\bbR^d;\bbR^N)$ is isometrically isomorphic to the space $C_b(L;\bbR)$, where $L$ is composed of $N$ disjoint copies of $\bbR^d$.
\end{proof}

Of course, one does not really need to assume that all operators $T_t$ (for $t > 0$) are compact in order to apply Theorem~\ref{thm:introduction}; it would suffice to assume that at least one operator $T_{t_0}$ is quasi-compact. However, the property that all $T_t$ are compact is quite a reasonable assumption in this setting since there are several sufficient criteria for this property available; we refer to \cite[Subsection~3.2]{Delmonte2011} for such conditions and refrain from stating them here explicitly.

However, let us illustrate the above result by the following simple concrete example, where the differential operator is a special case of the one considered in \cite[Section~4]{Delmonte2011}.

\begin{example} \label{ex:concrete-example-on-unbounded-domain}
	Consider the $\bbR^2$-valued evolution equation
	\begin{align}
		\label{eq:concrete-example-on-unbounded-domain}
		\begin{pmatrix}
			\dot u_1 \\ \dot u_2
		\end{pmatrix}
		= 
		\begin{pmatrix}
			\Delta u_1 \\ \Delta u_2
		\end{pmatrix}
		-
		\begin{pmatrix}
			(1+\norm{x}_2^2)^\beta \; x^T \nabla u_1 \\
			(1+\norm{x}_2^2)^\beta \; x^T \nabla u_2
		\end{pmatrix}
		+
		V(x)
		\begin{pmatrix}
			u_1 \\ u_2
		\end{pmatrix}
	\end{align}
	on $\bbR^d$, where $\beta > 0$ is a fixed real number and where $V(x)$ is given by
	\begin{align*}
		V(x) = 
		v(x)
		\begin{pmatrix}
			-1 & -1 \\
			-2 & -2
		\end{pmatrix}
		+
		w(x)
		\begin{pmatrix}
			-1 & -1 \\
			-1 & -1
		\end{pmatrix}
	\end{align*}
	for two functions $v,w\colon \bbR^d \to (0,\infty)$ that are bounded and locally $\alpha$-Hölder continuous with $\alpha \in (0,1)$.
	
	Examples of this type (in fact, of a more general type) are considered in~\cite[Section~4]{Delmonte2011}, where it is shown that this equation fits into the setting of the present section and that the solution semigroup of~\eqref{eq:concrete-example-on-unbounded-domain} is immediately compact on $C_b(\bbR^d; \bbR^2)$ \cite[Theorem~4.2]{Delmonte2011} (but note that the parameter $\alpha$ is used with different meaning there).
	
	It is not difficult to see that the matrix $V(x)$ is dissipative with respect to the $\ell^\infty$-norm on $\bbR^2$ for each $x \in \bbR^d$. Therefore, it follows from Theorem~\ref{thm:convergence-on-unbounded-domains} that the solution semigroup of~\eqref{eq:concrete-example-on-unbounded-domain} converges with respect to the operator norm on $C_b(\bbR^d; \bbR^2)$ as $t \to \infty$. The function $(\one, -\one)^T$ is an equilibrium, so the limit is non-zero for some initial values.
\end{example}

A few words about the choice of the potential $V$ in the preceding example are in order. The point about the sum of the two matrices in the definition of $V(x)$ is that it prevents the matrices $V(x)$ from being simultaneously diagonalisable (except for very simply choices of $v$ and $w$). In the case of simultaneous diagonalisability of the $V(x)$, we could transform the equation~\eqref{eq:concrete-example-on-unbounded-domain} into a form where both components decouple -- which means that we would essentially deal with two unrelated scalar equations.

\subsubsection*{Acknowledgements}

We would like to thank Markus Haase for pointing out to us the relation between the semigroups at infinity with respect to the strong and the operator norm topology explained in Remark~\ref{rem:embedded-semigroup-and-strong-op-topology}, and for providing an argument that simplified the proof of Theorem~\ref{thm:trivial-group-by-rational-spectrum}. 

We are also indebted to Abdelaziz Rhandi for remarking that the evolution equation~\eqref{eq:concrete-example-on-unbounded-domain} in Example~\ref{ex:concrete-example-on-unbounded-domain} decouples if all the matrices $V(x)$ are simultaneously diagonalisable.

\appendix

\section{On poles of operator resolvents} \label{section:poles-of-op-resolvents}

In the following proposition we briefly recall a result about poles of the resolvent of a linear operator. This result is needed in the proof of Proposition~\ref{prop:compactness-for-single-operators}.

\begin{proposition} \label{prop:pole-of-resolvent-by-resolvent-convergence}
	Let $T$ be a bounded linear operator on a complex Banach space $E$ and let $(\mu_j)_j$ be a net in the resolvent set of $T$ which converges to a number $\lambda \in \bbC$.  Then the following assertions hold:
	\begin{enumerate}[\upshape (a)]
		\item $\lambda \in \bbC \setminus \spec(T)$ if and only if the net $\big((\mu_j - \lambda)\Res(\mu_j,T)\big)_j$ converges to the zero operator.
		\item $\lambda$ is a spectral value of $T$ and a first order pole of the resolvent function $\Res(\argument, T)$ if and only if the net $\big((\mu_j - \lambda)\Res(\mu_j,T)\big)_j$ converges to a non-zero operator $P \in \calL(E)$. \\
		In this case, $P$ is the spectral projection associated with the pole $\lambda$.
	\end{enumerate}
\end{proposition}
\begin{proof}
	(a) The implication ``$\Rightarrow$'' is obvious, and the converse implication ``$\Leftarrow$'' follows from that well-known fact that, for every $\mu$ in the resolvent set of $T$, the norm of $\Res(\mu,T)$ is no less than $1/\dist(\mu, \spec(T))$ (where $\dist$ denotes the distance in the complex plane).
	
	(b) If $\lambda$ is a spectral value of $T$ and a first order pole of the resolvent, then the net $\big((\mu_j - \lambda)\Res(\mu_j,T)\big)_j$ obviously converges to the spectral projection associated with $\lambda$, and this spectral projection is non-zero. 
	
	Now assume conversely that the net $\big((\mu_j - \lambda)\Res(\mu_j,T)\big)_j$ converges to an operator $P \not= 0$. It then follows from~(a) that $\lambda$ is a spectral value of $T$; in particular, the elements of the net $(\mu_j)_j$ are eventually distinct from $\lambda$. Hence, it follows from the resolvent identity that
	\begin{align}
		\Res(\mu,T)P = \frac{P}{\mu - \lambda} \label{eq:resolvent-and-abel-projection}
	\end{align}
	for each $\mu$ in the resolvent set of $T$. From this we immediately obtain $P^2 = P$, i.e., $P$ is a projection; moreover, $P$ clearly commutes with $T$, so $T$ splits over the decomposition $E = \ker P \oplus PE$.
	
	It follows from~(a) that $\lambda$ is in the resolvent set of $T|_{\ker P}$. Moreover, we conclude from~\eqref{eq:resolvent-and-abel-projection} that $\lambda$ is a first order pole of the resolvent of $T|_{PE}$. Consequently, $\lambda$ is also a first order pole of the resolvent of $T$.
\end{proof}

\section{A few facts about nets} \label{section:universal-nets}

In this appendix we recall a few facts about nets and universal nets that are needed in the main text, in particular in Proposition~\ref{prop:characterisation-of-compact-and-non-empty}. Recall that a net $(x_j)$ in a set $X$ is called a \emph{universal net} if, for each $A \subseteq X$, the net is either eventually contained in $A$ or eventually contained in $X \setminus A$. If a subnet $(x_{j_i})$ of a net $(x_j)$ is a universal net, then we call $(x_{j_i})$ a \emph{universal subnet} of $(x_j)$. It follows from Zorn's lemma that every net has a universal subnet.

If $X$ is a topological Hausdorff space, then a subset $A \subseteq X$ is compact if and only if every universal net in $A$ converges to an element of $A$. In the following lemma we collect a few facts about metric spaces. For a proof we refer for instance to \cite[Theorem~B.3]{Glueck2019}, where these facts are given in a slightly more general topological setting.

\begin{lemma} \label{lemma:set-of-cluster-points}
	Let $(x_\alpha)_{\alpha \in I}$ be a net in a metric space $X$ and let
	\begin{align*}
		C \coloneqq \bigcap_{\beta \in I} \overline{\{x_\alpha \colon \alpha \geq \beta\}}
	\end{align*}
	be its set of cluster points. Consider the following assertions.
	\begin{enumerate}[\upshape (i)]
		\item The set $C$ is non-empty and compact.
		\item Each subnet of $(x_\alpha)_{\alpha \in I}$ has a convergent subnet.
		\item Each universal subnet of $(x_\alpha)_{\alpha \in I}$ converges.
		\item For each cofinal subsequence $(\alpha_n)_{n \in \bbN}$ in $I$ the sequence $(x_{\alpha_n})_{n \in \bbN}$ has a cluster point.
	\end{enumerate}
	Then $\text{\upshape (i)} \Leftarrow \text{\upshape (ii)}\Leftrightarrow \text{\upshape (iii)} \Rightarrow \text{\upshape (iv)}$. If, in addition, $I$ contains a cofinal sequence, then $\text{\upshape (iv)} \Rightarrow \text{\upshape (i)}$ as well.
\end{lemma}

\bibliographystyle{plain}
\bibliography{literature}

\end{document}